\documentclass[11pt,reqno,a4paper]{amsart}
\usepackage[english]{babel}
\usepackage[applemac]{inputenc}
\usepackage[T1]{fontenc}
\usepackage{palatino}
\usepackage{amsfonts}
\usepackage{amsmath}
\usepackage{amssymb}
\usepackage{amsthm}
\usepackage{graphicx}
\usepackage{wrapfig}
\usepackage{type1cm}
\usepackage[bf]{caption}
\usepackage{esint}
\usepackage{version}
\usepackage[colorlinks = true, citecolor = black, linkcolor = black, urlcolor = black, pdfstartview=FitH]{hyperref}
\usepackage[usenames]{color}
\usepackage{caption}
\usepackage{tikz}
\usepackage{bbding, wasysym}
\usepackage{verbatim}
\usepackage{psfrag}
\usepackage{textpos}
  
\newcommand{\R}{\mathbb{R}}

\newcommand{\N}{\mathbb{N}}

\newcommand{\E}{\mathbb{E}}

\renewcommand{\S}{\mathbb{S}}
\newcommand{\cB}{\mathcal{B}}

\newcommand{\cH}{\mathcal{H}}
\newcommand{\cL}{\mathcal{L}}
\newcommand{\Le}{\mathcal{L}}

\newcommand{\e}{\epsilon}

\newcommand{\cC}{\mathcal{C}}

\newcommand{\cQ}{\mathcal{Q}}

\newcommand{\dist}{\operatorname{dist}}
\newcommand{\card}{\operatorname{card}}

\renewcommand{\hom}{\operatorname{hom}}

\newcommand{\ulocd}{\operatorname{\overline{\dim}_{loc}}}
\newcommand{\llocd}{\operatorname{\underline{\dim}_{loc}}}
\newcommand{\udimloc}{\operatorname{\overline{\dim}_{loc}}}
\newcommand{\ldimloc}{\operatorname{\underline{\dim}_{loc}}}

\newcommand{\trap}{\operatorname{trap}}

\newcommand{\por}{\operatorname{por}}
\newcommand{\black}{\operatorname{black}}

\renewcommand{\emptyset}{\varnothing}
\renewcommand{\epsilon}{\varepsilon}
\renewcommand{\rho}{\varrho}
\renewcommand{\phi}{\varphi}

\newcommand{\1}{\mathrm{\mathbf{1}}}

\DeclareMathOperator{\proj}{proj}

\theoremstyle{plain}
\newtheorem{thm}{Theorem}
\newtheorem{theorem}[thm]{Theorem}
\newtheorem{lemma}{Lemma}
\newtheorem{proposition}{Proposition}

\newtheorem*{claim*}{Claim}

\theoremstyle{definition}
\newtheorem{definition}{Definition}

\newtheorem*{examples*}{Examples}
\newtheorem*{example*}{Example}
\newtheorem{remark}{Remark}

\newtheorem{notation}{Notation}
\newtheorem*{notations*}{Notations}
\newtheorem*{notation*}{Notation}

\numberwithin{equation}{section}
\numberwithin{thm}{section}
\numberwithin{lemma}{section}
\numberwithin{proposition}{section}
\numberwithin{cor}{section}
\numberwithin{claim}{section}
\numberwithin{definition}{section}
\numberwithin{conjecture}{section}
\numberwithin{example}{section}
\numberwithin{remark}{section}
\numberwithin{notations}{section}
\numberwithin{notation}{section}

\author{Tuomas Sahlsten, Pablo Shmerkin, and Ville Suomala}
\title{Dimension, entropy, and the local distribution of measures}

\thanks{TS acknowledges the support from the CoE in Analysis and Dynamics Research and Emil Aaltonen Foundation, PS was partially supported by a Leverhulme Early Career Fellowship, and VS was in part supported by the Academy of Finland, project $\sharp$126976.}

\address{Department of Mathematics and Statistics, P.O Box 68, FI-00014 University of Helsinki, Finland}\email{tuomas.sahlsten@helsinki.fi}
\address{Department of Mathematics, Faculty of Engineering and Physical Sciences, University of Surrey, Guildford, GU2 7XH, United Kingdom} \email{p.shmerkin@surrey.ac.uk}
\address{Department of Mathematical sciences, P.O Box 3000, FI-90014 University of Oulu, Finland} \email{ville.suomala@oulu.fi}


\begin{document}

\maketitle

\begin{abstract}
We present a general approach to the study of the local distribution of measures on Euclidean spaces, based on local entropy averages. As concrete applications, we unify, generalize, and simplify a number of recent results on local homogeneity, porosity and conical densities of measures.
\end{abstract}


\section{Introduction and statement of main results}

The \textit{upper-} and \textit{lower local dimensions} of a measure $\mu$ at a point $x\in\R^d$ are defined as
\begin{align*}
\ulocd(\mu,x) = \limsup_{r\searrow 0} \frac{\log\mu(B(x,r))}{\log r} \quad \text{and} \quad \llocd(\mu,x) = \liminf_{r\searrow 0} \frac{\log\mu(B(x,r))}{\log r},
\end{align*}
where $B(x,r)$ is the closed ball of center $x$ and radius $r$.

A central general problem in geometric measure theory consists in understanding the relation between local dimension and the distribution of measure inside small balls. Heuristically, at points where the local dimension is ``large'', one expects the measure to be ``fairly well distributed at many scales''. A number of concepts, such as porosity, conical densities, and homogeneity, have been introduced to make quantitative the notion of ``locally well distributed''. The relation between these (and other) such concepts and local dimension has been an active research area in the last decades (see below for references).

So far, each particular problem required an \textit{ad hoc} method to pass from information about the local distribution of measures, to information about mass decay or, in other words, local dimension (though there certainly is overlap among the ideas). The main contribution of this work is to show that in the Euclidean setting, the method of \textit{local entropy averages} provides a general framework that allows to unify, simplify, and extend all the previous results in the area. To the best of our knowledge, local entropy averages were first considered by Llorente and Nicolau \cite{LN04}. The basic result relating them to the local dimension of measures was proved by Hochman and Shmerkin in \cite{HochmanShmerkin11}, as a key step in bounding the dimension of projected measures. It was then further applied to the theory of porosity by Shmerkin in the recent paper \cite{Shmerkin11}.

The basic result on entropy averages is given in Proposition \ref{entropyaverages}. Even though the local entropy average formula may appear more complicated than the definition of local dimension, it is useful in many applications, since entropy takes into account the local distribution of measure. Moreover, as it is an average over scales, it is very effective to study properties which only hold on some proportion of scales.

On the other hand, entropy averages are defined in terms of dyadic partitions, while the geometric information one is interested in is usually in terms of Euclidean balls and cones. A random translation argument often allows to pass between one and the other.

This paper is organized as follows. In the rest of this section, we state the key entropy averages lemma, define the relevant geometric notions of homogeneity, porosity, and conical densities, and state our main results. In Section \ref{sec:preliminaries}, we collect a number of basic technical results that will be required later. Section \ref{sec:proofs} contains the proofs of the main results. Finally, we make some further remarks in Section \ref{sec:remarksandfurtheresults}.

\subsection{Local entropy averages}

\begin{notation}
Let $\cQ = \bigcup_{k \in \N} \cQ_k$ be the collection of all half open \textit{dyadic cubes} of $\R^d$, where $\cQ_k$ are the dyadic cubes $Q$ of side length $\ell(Q) = 2^{-k}$. Given $x \in \R^d$, we let $Q^{k,x}$ be the unique cube from $\cQ_k$ containing $x$. When $a \in \N$, we denote $Q' \prec_a Q$ if $Q \in \cQ_k$, $Q' \in \cQ_{k+a}$ and $Q' \subset Q$.

We keep to the convention that a \emph{measure} refers to a Borel regular locally finite outer measure. Since we are interested in local concepts, we shall further assume that our measures have compact support.
Lebesgue measure on $\R^d$ is denoted by $\cL^d$. For a measure $\mu$ on $[0,1)^d$, $x\in [0,1)^d$ and $k\in \N$, we denote by $\mu_{k,x}$ the normalized restriction of $\mu$ to $Q^{k,x}$. More precisely, if $\mu(Q^{k,x})=0$, then $\mu_{k,x}$ is the trivial measure; otherwise,
$$
\mu_{k,x}(A) = \frac{\mu(A)}{\mu(Q^{k,x})}\quad\text{for all Borel sets } A\subset Q^{k,x}.
$$
\end{notation}

For notational simplicity, in this article logarithms are always to base $2$.

\begin{definition}[Entropy]
The {\em entropy function} is the map $\phi : [0,1] \to \R$, $\varphi(t) := t\log(1/t)$.

If $\mu$ is a measure on $[0,1)^d$, $a \in \N$, and $Q$ is a cube with $\mu(Q)>0$, the $a$-\textit{entropy of $\mu$ in the cube $Q$} is defined by
$$H^a(\mu,Q) = \sum_{Q' \prec_a Q} \varphi\left(\frac{\mu(Q')}{\mu(Q)}\right).$$
\end{definition}

\begin{proposition}[Local entropy averages]
\label{entropyaverages}
Let $\mu$ be a measure on $[0,1)^d$ and $a \in \N$. Then for $\mu$ almost every $x \in [0,1)^d$,
$$\ulocd(\mu,x)  = \limsup_{N \to \infty} \frac{1}{N a} \sum_{k = 1}^N H^a(\mu,Q^{k,x});$$
and
$$\llocd(\mu,x) = \liminf_{N \to \infty} \frac{1}{N a} \sum_{k = 1}^N H^a(\mu,Q^{k,x}).$$
\end{proposition}

\begin{remark}Proposition \ref{entropyaverages} is a direct application of the law of large numbers for martingale differences. There are many variations of the exact statement, see for example \cite{HochmanShmerkin11,Shmerkin11,Shmerkin11b}. This particular formulation is due to Michael Hochman \cite{Hochman11}, see \cite[Theorem 5]{Shmerkin11b} for a proof. Llorente and Nicolau also considered local entropy averages, but they relied on the Law of the Iterated Logarithm rather than the Law of Large Numbers. As a result, they get sharper results but under stronger assumptions on the measure, such as dyadic doubling, see \cite[Corollary 6.2]{LN04}.
\end{remark}

\subsection{Local homogeneity}

\begin{definition}[Local homogeneity at a given point and scale]
The \textit{local homogeneity} of a measure $\mu$ at $x \in \R^d$ with parameters $\delta,\epsilon,r > 0$ is defined by
\begin{align*}
\hom_{\delta,\epsilon,r}(\mu,x) = \sup\{\card \cB : \,\,& \cB \text{ is a } (\delta r)\text{-packing of } B(x,r)\\
& \text{with } \mu(B) > \epsilon\mu(B(x,5r)) \text{ for all } B \in \cB\}.
\end{align*}
Here a $\delta$-\textit{packing} of a set $A\subset\R^d$ is a disjoint collection of balls of radius $\delta$ centred in $A$.
\end{definition}

It is important that $\mu(B)$ is not compared to $\mu(B(x,r))$ but to the measure of the enlarged ball $B(x,5r)$. The choice of the scaling parameter is not important as $5$ could be replaced with any $1<C<\infty$. See \cite[6.16]{KRS11}.

\begin{notation}For a measure $\mu$ on $\R^d$ and $0 \leq s \leq n$, we denote
\begin{align*}
\Omega^s(\mu) = \{ x: \ulocd(\mu,x) \ge s \} \quad \text{and} \quad \Omega_s(\mu) = \{ x: \llocd(\mu,x) \ge s \}.
\end{align*}
Moreover, write $[N] = \{1,\dots,N\}$ when $N \in \N$.
\end{notation}

In the above form the concept of homogeneity was introduced in \cite{KRS11} as a tool to study porosities and conical densities in Euclidean and more general metric spaces. Relations between homogeneity  and dimension have been considered also in \cite{BeliaevSmirnov02,JarvenpaaJarvenpaa05}.
An intuitive idea behind the homogeneity is the following; If the dimension of $\mu$ is larger than $s$ and if $\delta>0$ then, for typical $x$ and small $r>0$, one expects to find at least $\delta^{-s}$ disjoint sub-balls of $B(x,r)$ of diameter $\delta r$ with relatively large mass. In the main results of \cite{KRS11} this statement is made rigorous in a quantitative way.
Our main result concerning the relation between homogeneity and dimension is the following.

\begin{theorem}
\label{thm:homogeneity}
Let $0 < m < s < d$. Then there exist constants $p = p(m,s,d) > 0$ and $\delta_0 = \delta_0(m,s,d) > 0$ such that, for all $0 < \delta < \delta_0$, there exists $\epsilon = \epsilon(m,s,d,\delta) > 0$ with the following property: If $\mu$ is a measure on $\R^d$, then for $\mu$ almost every $x \in \Omega_s(\mu)$ and for all large enough $N \in \N$, we have
\begin{align}
\label{hommu}
\hom_{\delta,\epsilon,2^{-k}}(\mu,x) \geq \delta^{-m} \quad \text{for at least } pN \text{ values } k \in [N].
\end{align}
Moreover, for $\mu$ almost all $x \in \Omega^s(\mu)$, the estimate \eqref{hommu} holds for infinitely many $N$.
\end{theorem}

Theorem \ref{thm:homogeneity} is a stronger form of the statement \cite[Theorem 3.7]{KRS11}. The result in \cite{KRS11} yields a sequence of scales with large homogeneity for $\mu$ almost every point $x \in \Omega_s(\mu)$, but does not give any quantitative estimate on the amount of such scales. Moreover, Theorem \ref{thm:homogeneity} also yields information about the local homogeneity of measures at points in $\Omega^s(\mu)$.

Intuitively, our further applications on conical densities and porosities, should follow from Theorem \ref{thm:homogeneity} (this strategy was already used in \cite{KRS11}). With the result on conical densities (Theorem \ref{thm:meanconical} below) this is indeed the case as we apply a dyadic version of Theorem \ref{thm:homogeneity}. Concerning our result on the dimension of mean porous measures (Theorem \ref{thm:meanporosity}) we provide a proof which is independent of Theorem \ref{thm:homogeneity}, but the main idea is still in obtaining a homogeneity estimate.

\subsection{Conical densities}

\begin{notation}[Cones]
Let $m \in \{0,\dots,d-1\}$ and let $G(d,d-m)$ be the set of all $(d-m)$-dimensional linear subspaces of $\R^d$. Fix $x \in \R^d$, $r > 0$, $\theta \in \S^{d-1}$, $0 \leq \alpha \leq 1$ and $V \in G(d,d-m)$. Denote
\begin{align*}
H(x,\theta,\alpha) &= \{y \in \R^d : (y-x) \cdot \theta > \alpha |y-x|\},\\
X(x,V,\alpha) &= \{y \in \R^d : \dist(y-x,V) < \alpha|y-x|\},\\
X(x,r,V,\alpha) &= B(x,r) \cap X(x,V,\alpha).
\end{align*}
\end{notation}

The problem of relating the dimension of sets and measures to their distributions inside small cones has a long history beginning from the work of Besicovitch on the distribution of unrectifiable $1$-sets. The conical density properties of Hausdorff measures have been extensively studied by Marstrand \cite{Marstrand54}, Mattila \cite{Mattila88}, Salli \cite{Salli85} and others and have been applied e.g. in unrectifiability \cite{Mattila95} and removability problems \cite{MattilaParamonov95,Lorent03}. Analogous results for packing type measures were first obtained in \cite{Suomala05,KaenmakiSuomala11,KaenmakiSuomala08}.
Upper density properties of arbitrary measures have been considered in \cite{CKRS10,KRS11}. See also the survey \cite{Kaenmaki10}.

The goal in the theory of conical densities is to obtain information on the relative $\mu$ measure of the cones $X(x,r,V,\alpha)$ or $X(x,r,V,\alpha)\setminus H(x,\theta,\alpha)$, see Figure \ref{dimentpicCone} below.

\begin{figure}[h]
\begin{center}
\includegraphics[scale=0.6]{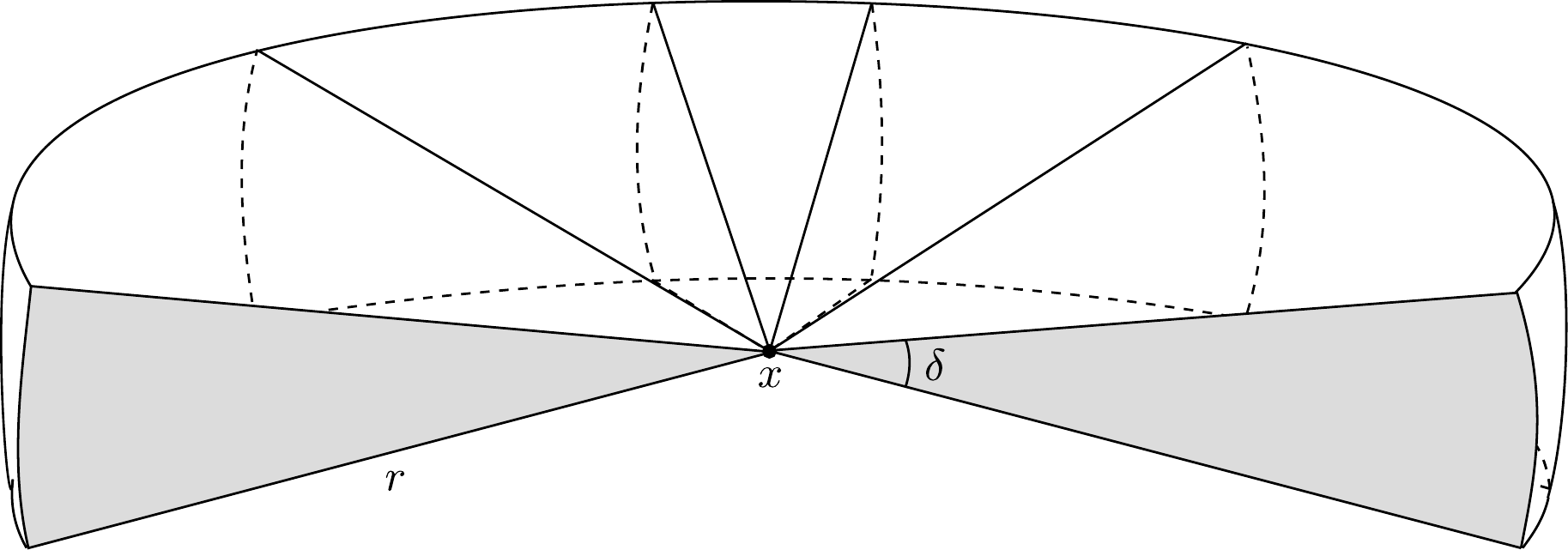}
\end{center}
\caption{Cone $X(x,r,V,\alpha) \setminus H(x,\theta,\eta)$ when $n = 3$, $m = 1$, $\alpha = \sin(\delta/2)$, and when the parameter $\eta$ of the cone $H(x,\theta,\eta)$ is close to $0$.}
\label{dimentpicCone}
\end{figure}

In order to get nontrivial lower bounds, we have to assume that the dimension of $\mu$ is larger than $m$, as the orthogonal complement of $V\in G(n,n-m)$ is $m$-dimensional.

We obtain the following result.
\begin{theorem}
\label{thm:meanconical}
Let $0 < m < s < d$, $m \in \N$ and $0 < \alpha < 1$. Then there exist constants $p = p(m,s,d,\alpha) > 0$ and $c = c(m,s,d,\alpha) > 0$ such that the following holds: If $\mu$ is a measure on $\R^d$, then for $\mu$ almost every $x \in \Omega_s(\mu)$ and for all large enough $N \in \N$ we have
\begin{align}
\label{conical}
\inf_{\theta \in \S^{d-1} \atop V\in G(d,d-m)}\frac{\mu(X(x,2^{-k},V,\alpha) \setminus H(x,\theta,\alpha))}{\mu(B(x,2^{-k}))} > c \quad \text{for at least } pN \text{ values } k \in [N].
\end{align}
For $\mu$ almost all $x \in \Omega^s(\mu)$ the estimate \eqref{conical} holds for infinitely many $N$.
\end{theorem}

Thus, if the dimension of $\mu$ is larger than $m$, the measure $\mu$ is rather uniformly spread out in all directions for a positive proportion of dyadic scales.

Theorem \ref{thm:meanconical} is a generalisation of \cite[Theorem 4.1]{CKRS10} and \cite[Theorem 5.1]{KRS11}. Concerning the statement about $\Omega_s(\mu)$, the results of \cite{CKRS10} (See Remark 4.7 in \cite{CKRS10}) yield that \eqref{conical} holds for infinitely many $N$ and our result strengthens this to \emph{all} large $N$. Regarding the estimate on $\Omega^s(\mu)$, the results of  \cite{CKRS10,KRS11} do not give any quantitative estimate on the amount of scales where \eqref{conical} holds.

\subsection{Porosity}

\begin{definition}[Porosity at a given point and scale]
Let $\ell\in[d]$. The $\ell$-\textit{porosity} of a set $A \subset \R^d$ at $x \in \R^d$ at scale $r > 0$
is
\begin{align*}\por_\ell(A,x,r) = \sup \{\rho > 0 : \,\,&\exists y_1,\dots,y_\ell \in \R^d, (y_i-x) \cdot (y_j-x) = 0, \\
& B(y_i,\rho r) \subset B(x,r) \setminus A\}.
\end{align*}
The $\ell$-\textit{porosity} of a measure $\mu$ at $x \in \R^d$ with parameters $r,\epsilon > 0$ and $\ell = 1,\dots,d$ is
\begin{align*}\por_\ell(\mu,x,r,\epsilon) = \sup \{\rho > 0 : \,\,&\exists y_1,\dots,y_\ell \in \R^d, (y_i-x) \cdot (y_j-x) = 0, \\
& B(y_i,\rho r) \subset B(x,r) \text{ and } \mu(B(y_i,\rho r)) \leq \epsilon\mu(B(x,r))\}.
\end{align*}
\end{definition}
Thus, $\por_\ell(A,x,r)$ is the supremum of all $\alpha$ such that we can find $\ell$ orthogonal holes in $A$, of relative size at least $\alpha$, inside the reference ball $B(x,r)$. The definition of $\por_\ell(\mu,x,r,\epsilon)$ is similar, except that ``holes'' are now measure theoretical, with a threshold given by $\epsilon$.

\begin{definition}[Mean porosity]\label{porodefi}
Let $0 < \alpha < 1/2$ and $0 < p \leq 1$. The measure $\mu$ is \textit{lower mean $(\ell,\alpha,p)$-porous} at $x \in \R^d$ if, for any $\epsilon > 0$ and large enough $N=N(\e,x) \in \N$,
\begin{equation}\label{porodef}
\por_\ell(\mu,x,2^{-k},\epsilon) > \alpha \quad \text{ for at least } pN \text{ values } k \in [N].
\end{equation}
We say that $\mu$ is \textit{upper mean $(\ell,\alpha,p)$-porous} at $x$ if for all $\epsilon>0$, the condition \eqref{porodef} holds for infinitely many $N\in\N$. Write
\begin{align*}
P_{\ell,\alpha,p} &= P_{\ell,\alpha,p}(\mu) = \{ x \in \R^d : \mu \text{ is lower mean } (\ell,\alpha,p)\text{-porous at } x\},\\
U_{\ell,\alpha,p} & =U_{\ell,\alpha,p}(\mu)=\{x \in \R^d : \mu \text{ is upper mean } (\ell,\alpha,p)\text{-porous at } x\}.
\end{align*}
\end{definition}

The study of the relationship between dimension of sets and their porosities originates from various size-estimates in geometric analysis. See e.g. \cite{Dolzenko67,Vaisala87,KoskelaRohde97}.
The connection between porosity of measures and their dimension has been studied e.g in \cite{EJJ00,BeliaevSmirnov02,KaenmakiSuomala08,BJJKRSS09,KRS11,Shmerkin11}. For further background and references, see the recent surveys \cite{Jarvenpaa10,Shmerkin11b} and also \cite{Rajala09a}. The next theorem unifies and extends many of the earlier results.
\begin{theorem}\label{thm:meanporosity} Let $0 < \alpha < \tfrac{1}{2}$, $0 < p \leq 1$, $\ell\in[d]$ and $\mu$ a measure on $\R^d$. For $\mu$ almost all $x\in P_{\ell,\alpha,p}$, we have
\begin{align}
\label{dimkporous}
\udimloc(\mu,x)\leq d-p\ell+\frac{c}{\log\frac{1}{1-2\alpha}},
\end{align}
where $c =c(d)<\infty$ is a dimensional constant. Moreover, for $\mu$ almost every $x\in U_{\ell,\alpha,p}$ it holds
\begin{align}
\label{dimkporousu}
\ldimloc(\mu,x)\leq d-p\ell+\frac{c}{\log\frac{1}{1-2\alpha}}\,.
\end{align}
\end{theorem}

The claim \eqref{dimkporous} gives a positive answer to the question \cite[Question 6.9]{KRS11}. For $\ell = 1$ 
it was already proven in \cite[Theorem 3.1]{BJJKRSS09}, but the method used in \cite{BJJKRSS09} relies heavily on the co-dimension being one and cannot be used when $\ell>1$. Moreover, the analogous result for porous (rather than mean porous) measures, was obtained in \cite[Theorem 5.2]{KRS11}.

\begin{remark}
The latter claim \eqref{dimkporousu} is the first nontrivial dimension estimate for upper mean porous sets or measures. All the previous works on the dimension of mean porous sets and measures deal with lower mean porosity.
\end{remark}

Theorem \ref{thm:meanporosity} is meaningful and gives the correct rate of convergence for $\alpha\to 1/2$ (it is not hard to see that $1/2$ is the largest possible value of $\alpha$ in a set of positive measure, see \cite[Remark (a) in p.4]{EJJ00}).
The small porosity situation (for $\ell=1$) was addressed in \cite{Shmerkin11} using entropy averages.

\section{Preliminaries}
\label{sec:preliminaries}

\subsection{Maximizing entropy} This section is devoted to maximize entropy in a symbolic setting with certain boundary conditions that will appear in our applications.

\begin{definition}
The \textit{entropy} of the $K$-tuple $(p_1,\dots,p_K)$, $0 \leq p_i \leq 1$, is
$$H(p_1,\dots,p_K) := \sum_{i = 1}^K p_i \log (1/p_i).$$
\end{definition}

We remark that it is {\em not} assumed that $\sum_{i=1}^K p_i =1$. Recall the following estimate, which is an easy consequence of Jensen's inequality:

\begin{lemma}[Log sum inequality]
\label{lemma:logsum}
Let $a_1,\dots,a_K$ and $b_1,\dots,b_K$ be non-negative reals. Then
$$\sum_{i = 1}^K a_i \log \frac{a_i}{b_i} \geq \Big(\sum_{i = 1}^K a_i\Big) \log \frac{\sum_{i = 1}^K a_i}{\sum_{i = 1}^K b_i}.$$
\end{lemma}

\begin{lemma}
\label{trivialentropy}
If $0 < c \leq 1$, then
\begin{align*}
\max_{p_1,\dots,p_K \geq 0 \atop \sum_{i = 1}^K p_i = c} H(p_1,\dots,p_K)= H\Big(\frac{c}{K},\dots,\frac{c}{K}\Big) = c\log(K/c).
\end{align*}
\end{lemma}

\begin{proof}
Apply the log sum inequality with $a_i = p_i$ and $b_i = 1/K$.
\end{proof}

\begin{lemma}\label{maximalentropy} Let $M \in \N$, $0 < \epsilon < \tfrac{1}{2M}$ and $n \in \{ 0, \dots, M-1\}$. Then
\begin{align}
\label{maxentvalue}
h^{M}_{\epsilon,n} := \max H|_{\Delta_{\epsilon,n}^M} = (1-n\epsilon)\log\Big(\frac{M-n}{1-n \epsilon}\Big)+n\epsilon\log(1/\epsilon),
\end{align}
where
\begin{align*}
\Delta_{\epsilon,n}^M := \Big\{(q_1,\dots,q_{M-n},p_1,\dots,p_n) : \sum_{j = 1}^{M-n} q_j = 1 - \sum_{i = 1}^n p_i, \, q_j \geq 0 \text{ and }0 \leq p_i \leq \epsilon\Big\}.
\end{align*}
\end{lemma}

\begin{proof}
If $0 \leq p_i \leq \epsilon$, $i\in[n]$ are fixed, Lemma \ref{trivialentropy} (applied with $K = M-n$ and $c = 1-\sum_{i = 1}^n p_i$) implies that
\begin{equation}\label{j}
\max_{q_j\ge0\text{ and }\sum_{j = 1}^{M-n} q_j = 1-\sum_{i = 1}^n p_i}H(q_1,\dots,q_{M-n})=\Big(1-\sum_{i = 1}^n p_i\Big)\log\Big(\frac{M-n}{1-\sum_{i = 1}^n p_i}\Big).
\end{equation}
Moreover,
\begin{equation}\label{m}
\max_{0\le p_1,\ldots,p_n\le\epsilon}\Big(1-\sum_{i = 1}^n p_i\Big)\log\Big(\frac{M-n}{1-\sum_{i = 1}^n p_i}\Big)=(1-n\epsilon)\log\Big(\frac{M-n}{1-n\epsilon}\Big).
\end{equation}
Indeed, since $p_i \leq \epsilon < \tfrac{1}{2M}$ and $n \leq M-1$, we have $1-\sum_{i = 1}^n p_i \geq 1-n\epsilon > \tfrac12$.
The map $g(t) = (1-t)\log\tfrac{M-n}{1-t}$, $t \in [\tfrac12,1]$, is decreasing, so the maximum value of $g$ on $[1-n\epsilon,1]$ is attained at the left boundary point $1-n\epsilon$, that is, when each $p_i = \epsilon$.

Lemma \ref{trivialentropy} also implies
\begin{equation}\label{v}
\max_{0\le p_1,\ldots,p_n\le\epsilon}H(p_1,\ldots,p_n)=H(\epsilon,\ldots,\epsilon)=n\epsilon\log(1/\epsilon).
\end{equation}
Combining \eqref{j}, \eqref{m} and \eqref{v}, it follows that
\begin{align*}\max H|_{\Delta_{\epsilon,n}^M}
& = \max_{0 \leq p_1,\dots,p_n \leq \epsilon} \max_{q_1,\dots,q_{M-n} \geq 0\atop  \sum_{j = 1}^{M-n} q_j = 1-\sum_{i = 1}^n p_i } \Big(H(q_1,\dots,q_{M-n})+H(p_1,\dots,p_n)\Big)\\
& = \max_{0 \leq p_1,\dots,p_n \leq \epsilon} \Big(1-\sum_{i = 1}^n p_i\Big)\log\Big(\frac{M-n}{1-\sum_{i = 1}^n p_i}\Big)+ H(p_1,\dots,p_n)\\
& = (1-n\epsilon)\log\Big(\frac{M-n}{1-n \epsilon}\Big)+n\epsilon\log(1/\epsilon),
\end{align*}
as claimed.
\end{proof}

\subsection{Abundance of doubling scales}

If $\mu$ is a measure on $\R^d$ and $\omega\in \R^d$, we denote by $\mu_\omega$ the translation of $\mu$ by $\omega$, i.e. $\mu_\omega(A)=\mu(A+\omega)$ for any Borel set $A$. The trick of randomly translating the dyadic lattice has turned out to be extremely useful in many situations. For instance, Nazarov, Treil, and Volberg \cite{NTV03} use it in the proof of their non-homogeneous $Tb$-theorem. In \cite{Shmerkin11}, this technique is for the first time used to study dimension of porous measures.
In this section, we prove the following result on the existence of many ``doubling scales''. Given $K>0$ and a cube $Q$ we denote by $KQ$ the cube with the same center as $Q$ and $K$ times the side length.

\begin{lemma}
\label{lemma:meandoubling}
For any $0 < p < 1$ and $K\in \N$ there exists $c' = c'(d,p,K) > 0$ such that the following holds:
Let $\mu$ be a
measure on $[0,1/2)^d$. Then for $\cL^d$ almost every $\omega\in[0,1/2)^d$,
$$\liminf_{N\rightarrow\infty} \frac{1}{N}\card\{ k \in [N] : \mu_\omega(Q^{k,x}) \geq c'\mu_\omega(K Q^{k,x})\} \ge p$$
for $\mu_\omega$ almost every $x \in \R^d$.
\end{lemma}

For this, we need a few lemmas first.

\begin{notation} If $2^{a-1}> K$ and $Q\in \cQ_k$, let
$$U(Q)=U_{K,a}(Q) = \{x\in Q\,:\,\dist(x,\partial Q)>K 2^{-(k+a)}\},$$
where $\partial Q$ is the boundary of the cube.
Hence $U(Q)$ is obtained from $Q$ by removing the $K$ outer layers of subcubes of $Q$ of generation $k+a$.
\end{notation}

\begin{lemma}\label{lemma:random0}
If $2^{a-1}>K$ and $\mu$ is a measure on $[0,1/2)^d$, then for $\cL^d$ almost every $\omega\in[0,1/2)^d$, the translated measure $\mu_\omega$ satisfies
\begin{equation}\label{nottooclosetoboundary}
\lim_{N\rightarrow\infty}\frac{1}{N}\card\{k\in[N]\,:\,Q^{k+a,x}\subset U(Q^{k,x})\}=(1-K 2^{1-a})^d\quad\text{for } \mu_\omega \text{ almost all } x.
\end{equation}
\end{lemma}

\begin{proof}
We use the argument from \cite[Lemma 4.3]{Shmerkin11}. By the law of large numbers if $x\in\R^d$, then
$$\lim_{N\rightarrow\infty}\frac{\card\{k\in[N]\,:\,Q^{k+a,x-\omega}\subset U(Q^{k,x-\omega})\}}{N}=(1-K 2^{1-a})^d$$
for $\cL^d$ almost all $\omega\in[0,1/2)^d$
(in other words, $\cL^d$ almost all points $x-\omega$ are normal with respect to the standard $2^a$-adic grid).
From this the claim follows by applying Fubini's theorem to $\mu\times\Le^d|_{[0,1/2)^d}$.
\end{proof}

The following lemma is standard. As we have not been able to find this exact statement in the literature, and the proof is elementary, we include it for completeness.

\begin{lemma} \label{lemma:boundeddyadicdim}
If $\mu$ is a
measure  on $[0,1)^d$, then
\[
\limsup_{N\to\infty} \frac{\log\mu(Q^{N,x})}{-N} \le d
\]
for $\mu$ almost all $x$.
\end{lemma}
\begin{proof}
Suppose the claim does not hold. Then there are $\e>0$ and a set $A$ of positive $\mu$ measure such that the following holds: if $x\in A$, then
\[
\limsup_{N\to\infty} \frac{\log\mu(Q^{N,x})}{-N} \ge d+\e.
\]
Take a closed subset $A_0$ of $A$ with $\mu(A_0)>0$. For any $n_0\in\N$ and each $x\in A_0$, we can find $N(x)\ge n_0$ such that $\mu(Q^{N(x),x}) < 2^{-N(x) (d+\e)}$. Since $A_0$ is compact, we can cover it by finitely many of the $Q^{N(x),x}$. Writing $n_1$ for the maximum of the $N(x)$ on this finite set, an easy inductive argument, working from $n_1$ down to $n_0$, shows that $\mu(A_0\cap Q)\le 2^{- n_0 (d+\e)}$ for all cubes $Q\in\mathcal{Q}_{n_0}$, and therefore
\[
\mu(A_0) \le 2^{n_0 d} 2^{- n_0 (d+\e)} = 2^{-\e n_0}.
\]
As $n_0$ was arbitrary, $\mu(A_0)=0$. This contradiction finishes the proof of the lemma.
\end{proof}

\begin{lemma}
\label{lemma:dyadicdoubling}
If $a \in \N$, $0<p<1$ and $\mu$ is a 
measure on $[0,1)^d$, then
\begin{equation}\label{cor:localdim}
\liminf_{N\rightarrow\infty} \frac1N \card\{k\in[N]\,:\,\mu(Q^{k+a,x})\geq c\mu(Q^{k,x})\}\geq p
\end{equation}
for $\mu$ almost all $x\in\R^d$, where $c = c(d,p,a) > 0$.
\end{lemma}

\begin{proof}
The proof is similar to \cite[Lemma 2.2]{CKRS10}. To simplify notation, we assume that $\mu([0,1)^d)=1$. Choose $c = c(d,p,a) :=2^{-2ad/(1-p)}$. Let $x \in \R^d$ be any point where \eqref{cor:localdim} fails. Then there are arbitrarily large integers $N$ with
$$\card\{k\in[N]\,:\,\mu(Q^{k+a,x})\geq c\mu(Q^{k,x})\} < pN.$$
For such an $N$,
\begin{align*}\left(\mu(Q^{N+a,x})\right)^a \leq \prod_{\ell = 1}^a \frac{\mu(Q^{N+\ell,x})}{\mu(Q^{\ell,x})} = \prod_{k = 1}^N \frac{\mu(Q^{k+a,x})}{\mu(Q^{k,x})} < c^{(1-p)N}.
\end{align*}
Hence
$$\frac{\log\mu(Q^{N+a,x})}{-(N+a)} \geq \frac{\log c^{(1-p)N/a}}{-(N+a)}=\frac{-2d N}{-(N+a)} = \frac{2d N}{N+a},$$
which yields
$$\limsup_{N\rightarrow\infty}\frac{\log(\mu(Q^{N,x}))}{-N} \geq 2d.$$
In view of Lemma \ref{lemma:boundeddyadicdim}, \eqref{cor:localdim} can fail only in a $\mu$ null set, as claimed.
\end{proof}

\begin{proof}[Proof of Lemma \ref{lemma:meandoubling}]
Let $\omega\in [0,1/2)^d$ be such that \eqref{nottooclosetoboundary} holds. By Lemma \ref{lemma:random0} this is the case for $\cL^d$ almost all $\omega\in [0,1/2)^d$.
Choose $a = a(p,K)$ such that $2^{a-1}>K$ and $(1-K2^{1-a})^{d}>p$. Let $p_a = ((1-K2^{1-a})^d + p)/2$. Then $0 < p' := 1+p-p_a < 1$. Let $x \in \R^d$ be a point such that
\begin{align}
\label{a}
Q^{k,x}\subset U(Q^{k-a,x})
\end{align}
for at least $p_a N$ values $k \in [N]$; and
\begin{align}
\label{b}
\mu_\omega(Q^{k,x})\geq c'\mu_\omega(Q^{k-a,x})
\end{align}
for at least $p' N$ values $k \in [N]$, where $c' = c(d,p',a) > 0$ is the constant from Lemma \ref{lemma:dyadicdoubling}.
By Lemmas \ref{lemma:random0} and \ref{lemma:dyadicdoubling}, for $\mu_\omega$ almost all $x$, the conditions \eqref{a}, \eqref{b} are satisfied when $N$ is large.
Then \eqref{a} and \eqref{b} are simultaneously satisfied for at least $pN$ values $k \in [N]$.
Consider such $k \in [N]$. By \eqref{a} and the definition of $U(Q^{k-a,x})$,
we have $K Q^{k,x} \subset Q^{k-a,x}$, see the Figure \ref{dimentpicSetUQ} below. Now \eqref{b} yields
$\mu_\omega(Q^{k,x}) \geq c'\mu_\omega(K Q^{k,x})$
for at least $pN$ values $k \in [N]$ and the claim follows.
\end{proof}

\begin{figure}[h]
\begin{center}
\includegraphics[scale=0.55]{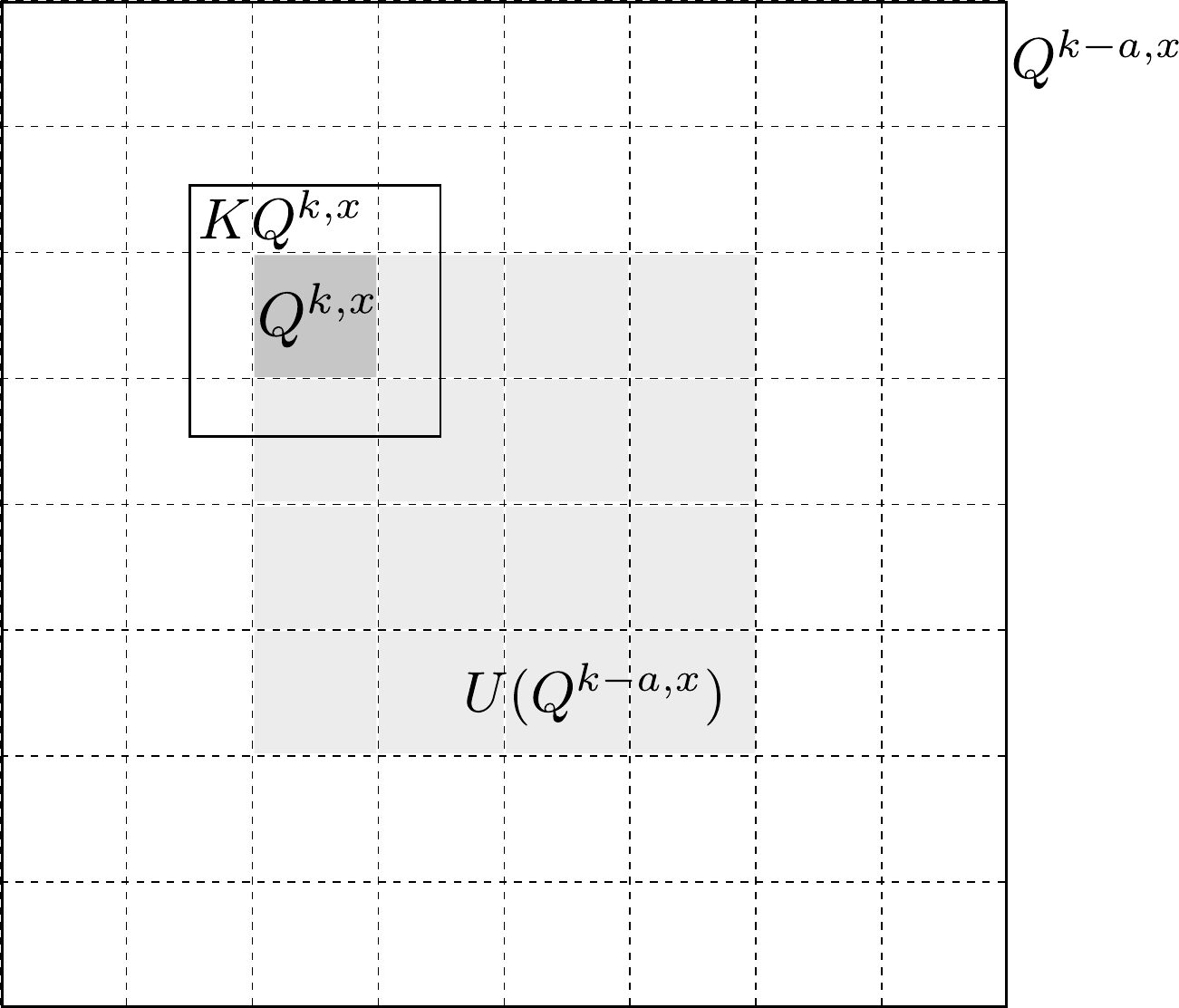}
\end{center}
\caption{We choose the number $a \in \N$ in the proof of Lemma \ref{lemma:meandoubling} so large that whenever a cube $Q \prec_a Q^{k-a,x}$ is contained in $U(Q^{k-a,x})$, then the enlarged cube $KQ \subset Q^{k-a,x}$. In the picture the light gray cubes forms the set $U(Q^{k-a,x})$, the darker gray cube $Q^{k,x} \subset U(Q^{k-a,x})$ and the constants $K = 2$ and $a = 3$.}
\label{dimentpicSetUQ}
\end{figure}

\subsection{Labeling cubes}

\begin{lemma}\label{lemma:blackandwhite}
Let $\mu$ be a Borel probability measure on $[0,1)^d$ and $a \in \N$. Suppose that each $Q \in \cQ$ is labeled either as 'black' or 'white'. Write
$$Q_{\black} := \bigcup_{Q' \prec_a Q \emph{\text{ is black}}} Q'.$$
Then
$$\lim_{N \to \infty} \frac{1}{N} \sum_{k = 1}^N \Big(\mu_{k,x}(Q^{k,x}_{\black}) - \1_{Q^{k,x} \text{ \emph{is black}}}\Big)= 0$$
for $\mu$ almost every $x \in [0,1)^d$.
\end{lemma}

\begin{proof}
For $k\ge a$, let
$$
f_k(x) = \mu_{k-a,x}(Q^{k-a,x}_{\black}), \quad\text{and}\quad g_k(x) = \1_{Q^{k,x} \text{ is black}}.
$$
Write $\cB_k$ for the $\sigma$-algebra generated by $\cQ_k$. Then $f_k$ is $\cB_{k-a}$ measurable, and $g_k$ is $\cB_k$ measurable. Moreover, the conditional expectation
$$\E_\mu (g_k|\cB_{k-a})(x) = \frac{\mu(Q^{k-a,x}_{\black})}{\mu(Q^{k-a,x})} = f_k(x).$$

This shows that, for each $j\in[a-1]$, the sequence $(f_{\ell a+j}-g_{\ell a+j})_{\ell \in \N}$ is a uniformly bounded martingale difference sequence (with respect to the filtration $\{\cB_{\ell a+j}\}_{\ell=0}^\infty$). By the law of large numbers for martingale differences (see \cite[Theorem 3 in Chapter VII.9]{Feller71}),
$$
\lim_{N \to \infty} \frac{1}{N} \sum_{\ell = 1}^N (f_{\ell a+j}(x) - g_{\ell a+j}(x)) = 0 \quad \text{for } \mu \text{ almost every } x \in [0,1)^d.
$$
Adding over all $j$, we deduce that
$$\lim_{N \to \infty} \frac{1}{N} \sum_{k = 1}^N (f_{k}(x) - g_{k}(x)) = 0 \quad \text{for } \mu \text{ almost every } x \in [0,1)^d.$$
The statement follows since
$$\frac{1}{N}\sum_{k = 1}^N (f_{k+a}(x) - g_k(x)) - \frac{1}{N}\sum_{k = 1}^N (f_k(x) - g_k(x)) \longrightarrow 0, \quad \text{as }N \to \infty,$$
for every $x$.
\end{proof}

\section{Proofs of the main results}
\label{sec:proofs}

\subsection{General outline}

Although the geometric details differ, the strategy of proof of Theorems \ref{thm:homogeneity}, \ref{thm:meanconical} and \ref{thm:meanporosity} follows a unified pattern, which we can summarize as follows:

\begin{enumerate}
\item Firstly, note that both the hypotheses and the statements are local and translation-invariant, so the measure $\mu$ can be assumed to be supported on $[0,1/2)^d$ and translated by a random vector in $[0,1/2)^d$.
\item Translate the geometric concept under study (homogeneity, conical densities, porosity) into a dyadic analogue.
\item Show that the validity of the original condition at a proportion $p$ of dyadic scales implies the validity of the corresponding dyadic version at a proportion $p'$ of scales, with $p'$ arbitrarily close to $p$. (This step usually depends on Lemma \ref{lemma:meandoubling}, and hence explains the initial random translation).
\item Use the geometric hypothesis (in its dyadic version) to obtain an estimate for the entropy $H(\mu, Q^{k,x})$ at points $x$ and scales $k$ such that the condition is verified (for example, at porous scales).
\item Conclude with a bound on the dimension of the measure from local entropy averages  (Proposition \ref{entropyaverages}).
\end{enumerate}
It may be useful to keep these steps in mind while going through the details of the proofs.

We remark that this general strategy was already used in \cite{Shmerkin11} (see also \cite{Shmerkin11b}). However, the geometric arguments in our current applications are rather more involved at all steps.

\subsection{Proof of Theorem \ref{thm:homogeneity}}

Throughout this section, constants $d\in\N$, $0<m<s<d$, and a measure $\mu$ on $[0,1)^d$ are fixed.

Theorem \ref{thm:homogeneity} will be proved via a corresponding dyadic formulation. For this we require a dyadic version of homogeneity:

\begin{definition}[Dyadic homogeneity]
Fix $a \in \N$. The \textit{dyadic $a$-homogeneity} of a measure $\mu$ at $x \in \R^d$ with parameters $\epsilon > 0$ and $k \in \N$ is defined by
$$\hom^a_{\epsilon,k}(\mu,x) = \card\{Q \prec_a Q^{k,x} : \mu(Q) > \epsilon\mu(Q^{k,x})\}.$$
\end{definition}

The required dyadic version of Theorem \ref{thm:homogeneity} is given in the next proposition.

\begin{proposition}
\label{prop:dyadichomo}
There exist constants $p_0 = p_0 (m,s,d) > 0$
such that for every $a \in\N$ there exists $\epsilon_0 = \epsilon_0(m,s,d,a) > 0$ with the following property: for $\mu$ almost every $x \in \Omega_s(\mu)$, and for all large enough $N \in \N$ we have
\begin{align}
\label{dyahommu}
\hom^a_{\epsilon_0,k}(\mu,x) > 2^{am} \quad \text{for at least } p_0 N \text{ values } k \in [N].
\end{align}
For $\mu$ almost all $x \in \Omega^s(\mu)$ the estimate \eqref{dyahommu} holds for infinitely many $N$.
\end{proposition}

First we need a more quantitative statement. We use the notation from Lemma \ref{maximalentropy}:

\begin{lemma}
\label{lemma:homogeneitydimension}
Let $0 < q < 1$, $a \in \N$, $0 < \epsilon < 2^{-ad-1}$. Suppose that $A \subset [0,1)^d$ has the following property: for every $x \in A$ there exist infinitely many $N \in \N$ such that
\begin{align}
\label{dimhomineq}
\hom_{\epsilon,k}^a(\mu,x) \le 2^{am} \quad \text{ for at least } qN \text{ values of } k \in [N].
\end{align}
Then
\begin{align}
\label{locineq1}
\llocd(\mu,x) \leq \frac{q h^{2^{ad}}_{\epsilon,2^{ad}-2^{am}}}{a}+(1-q)d =: s'(m,d,q,a,\epsilon)
\end{align}
for $\mu$ almost every $x \in A$. Moreover, if the inequality \eqref{dimhomineq} holds for all large enough $N$, then
\begin{align}
\label{locineq2}
\ulocd(\mu,x) \leq s'
\end{align}
for $\mu$ almost every $x \in A.$
\end{lemma}

\begin{proof}
Let $x \in A$ and choose $N \in \N$ such that \eqref{dimhomineq} holds. By \eqref{dimhomineq} there exist distinct $k_1,\dots,k_{\lceil qN\rceil} \in [N]$ with $\hom_{\epsilon,k_\ell}^a(\mu,x) \le 2^{am}$ for each $\ell\in\bigl[\lceil qN \rceil\bigr]$.
For any such $\ell$ we may choose $Q_1,Q_2,\dots,Q_{2^{ad}-\lfloor 2^{am}\rfloor} \prec_a Q^{k_\ell,x}$ such that $\mu(Q_i) \leq \epsilon\mu(Q^{k_\ell,x})$ for each $i$.
A direct application of Lemma \ref{maximalentropy} with $M=2^{ad}$ and $n=2^{ad}-\lfloor 2^{am}\rfloor$ then implies
$$H^a(\mu,Q^{k_\ell,x})\leq h^{2^{ad}}_{\epsilon,2^{ad}-2^{am}}, \quad \ell\in \bigl[\lceil qN\rceil\bigr].$$

Moreover, applying Lemma \ref{trivialentropy} with $K = 2^{ad}$ and $c = 1$, we have
$H^a(\mu,Q^{k,x}) \leq ad$ for $k \in [N] \setminus \{k_1,\dots,k_{\lceil qN\rceil}\}$.
Hence
\begin{align*}
\sum_{k = 1}^N H^a(\mu,Q^{k,x}) &= \sum_{\ell = 1}^{\lceil qN\rceil} H^a(\mu,Q^{k_\ell,x}) + \sum_{k \in [N]\atop k \notin \{k_{1},\dots,k_{\lceil qN\rceil}\}} H^a(\mu,Q^{k,x}) \\
& \leq qNh^{2^{ad}}_{\epsilon,2^{ad}-2^{am}} + (1-q)N ad.
\end{align*}
Dividing both sides by $N a$, we have
$$\frac{1}{N a}\sum_{k = 1}^N H^a(\mu,Q^{k,x}) \leq s'.$$
As $x \in A$, there are infinitely many $N$ such that the above holds. Hence \eqref{locineq1} is just an application of local entropy averages (Proposition \ref{entropyaverages}).

Similarly, if the above is satisfied for \textit{all} large enough $N \in \N$, then local entropy averages implies the stronger estimate \eqref{locineq2}.
\end{proof}

\begin{proof}[Proof of Proposition \ref{prop:dyadichomo}] Let $a \in \N$,
$p_0:=\tfrac{s-m}{2(d-m)}$, $t= (s-p_0d)/(1-p_0)$. Then
$\delta := a(t-m) > 0$.
Since (recall \eqref{maxentvalue})
$$\lim_{\epsilon \searrow 0} h^{2^{ad}}_{\epsilon,2^{ad}-2^{am}} = am,$$
we can choose $0 < \epsilon_0 = \epsilon_0(m,s,d,a) < 2^{-ad-1}$ such that
$$h^{2^{ad}}_{\epsilon_0,2^{ad}-2^{am}} < am + \delta = at.$$
Then by the definition of $s'$ (recall \eqref{locineq1}) we have
\begin{align}
\label{dimensionbound}
s'=s'(m,d,1-p_0,a,\epsilon_0) = \frac{(1-p_0) h^{2^{ad}}_{\epsilon_0,2^{ad}-2^{am}}}{a}+p_0d < (1-p_0)t+p_0d = s.
\end{align}

Assume that the first statement of Proposition \ref{prop:dyadichomo} fails. Then
there is a Borel set $A\subset\Omega_s(\mu)$ with $\mu(A)>0$, such that for every $x\in A$, there are infinitely many $N$ with
$$\hom_{\epsilon_0,k}^a(\mu,x) \le 2^{am} \quad \text{ for at least } (1-p_0)N \text{ values of } k \in [N].$$
Lemma \ref{lemma:homogeneitydimension} \eqref{locineq1} with $q = 1-p_0$, and \eqref{dimensionbound} then yield
$\llocd(\mu,x) \leq s' < s$
for $\mu$ almost every $x \in A$.
This is impossible since $\mu(A)>0$ and $A\subset\Omega_s(\mu)$.

The proof for $x\in\Omega^s(\mu)$ is completely analogous, using \eqref{locineq2} in place of \eqref{locineq1}.
\end{proof}

\begin{proof}[Proof of Theorem \ref{thm:homogeneity}]

We first fix various constants. Let $p_0 = p_0(m,s,d)>0$ be the constant from Proposition \ref{prop:dyadichomo} and
$p= p_0/2 > 0$. Let $t = (m+s)/2$
and let $2\le a_0 = a_0(m,s,d)$ be so large that
\begin{equation}\label{choiceof_a}
2^{a_0t}/C_d \geq 2^{a_0m}.
\end{equation}
where $C_d$ is a large dimensional constant (choosing $C_d=17 d^{d/2}$ will do).
Consider $0<\delta<\delta_0 := 2^{1-a_0}$.
Denote by $h=h(d)$ the smallest integer such that $2^{h} \ge \sqrt{d}$. Fix $a \geq a_0$ such that
\begin{equation}\label{choiceofdelta}
2^{-a}\le\delta<2^{1-a}.
\end{equation}
Let $K$ be a constant to be chosen later, depending only on $d$. By localizing, rescaling, and translating, we may assume that $\mu$ is supported on $[0,1)^d$ and satisfies the conclusion of Lemma \ref{lemma:meandoubling} with this value of $K$, i.e.
\begin{equation} \label{eq:apply-mean-doubling-homogeneity}
\liminf_{N\rightarrow\infty}\frac1N \card\{ k \in [N] : \mu(Q^{k,x}) \geq c'\mu(K Q^{k,x})\} >1-p\quad\text{for }\mu\text{ almost every } x,
\end{equation}
where $c'=c(d,1-p/2,K)$.
Further, let $\epsilon = c'\epsilon_0(t,s,d,a)$, where $\epsilon_0(t,s,d,a)$ is the value provided by Proposition \ref{prop:dyadichomo}.

Then, Proposition \ref{prop:dyadichomo} together with \eqref{eq:apply-mean-doubling-homogeneity} implies that for $\mu$ almost all $x \in\Omega_s(\mu)$ and all large $N \in \N$, it holds that
\begin{align}
\label{1}
\hom_{\epsilon_0,k+h}^a(\mu,x) \geq 2^{at}
\end{align}
for at least $2pN$ values $k \in [N]$, and
\begin{align}
\label{2}
\mu(Q^{k+h,x})\geq c'\mu(K Q^{k+h,x})
\end{align}
for at least $(1-p) N$ values $k \in [N]$. The assumptions \eqref{1} and \eqref{2} are then simultaneously satisfied for at least $pN$ values $k \in [N]$. Fix one such $k \in [N]$.

Denote
$$\cH_k := \{Q \prec_a Q^{k+h,x} : \mu(Q) > \epsilon_0\mu(Q^{k+h,x})\}.$$
We define a packing $\cB_k$ of $B(x,2^{-k})$ as follows. First pick $Q_1 \in \cH_k$, and remove all cubes $Q \in \cH_k$ which intersect the ball $B(x_{Q_1},\delta 2^{1-k})$. The estimates \eqref{choiceofdelta}
and a simple geometric inspection show there are at most $C_d$ such cubes (this is the only property of the constant $C_d$ that we use). Now choose $Q_2$ from the remaining cubes in $\cH_k$ and do the same. Hence each time we choose $Q_j$, we remove at most $C_d$ cubes from $\cH_k$. By \eqref{1}, we have $\card \cH_k \geq 2^{at}$, so this process stops when we have chosen the cube $Q_L$ with $L\ge 2^{at}/C_d$.

\begin{figure}[h]
\begin{center}
\includegraphics[scale=0.6]{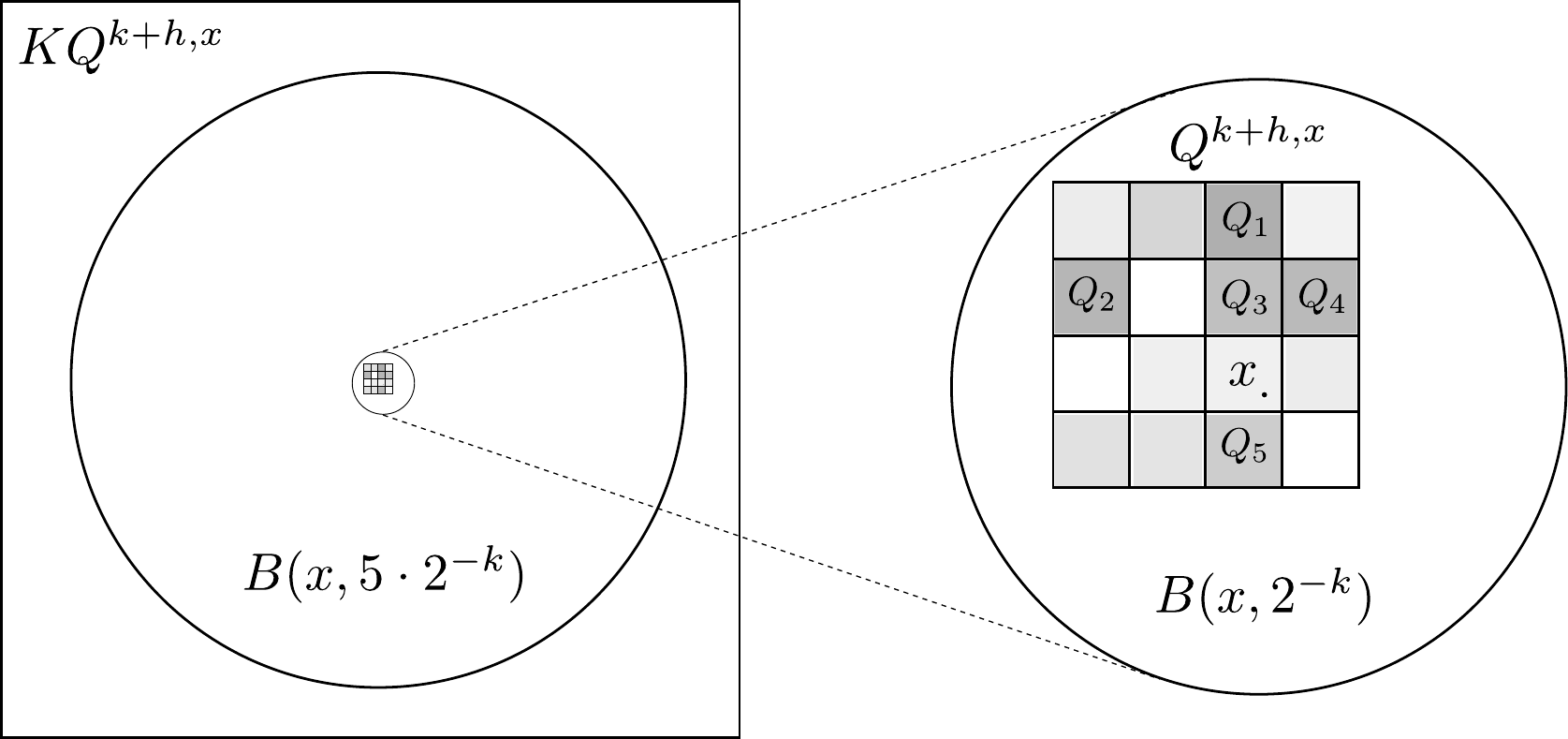}
\end{center}
\caption{We choose the constant $h$ such that $Q^{k+h,x}$ is contained in the reference ball $B(x,2^{-k})$, and then the constant $K$ such that $B(x,5\cdot 2^{-k}) \subset KQ^{k+h,x}$. The darker small cubes $Q_j\prec_a Q^{k+h+a,x}$ in the picture form the set $\cH_k$: each of them have at least $\epsilon_0$-portion of the $\mu$-mass of the cube $Q^{k+h,x}$. This allows us to control the proportion of $\mu$-masses of the balls $B_j \supset Q_j$ inside $B(x,5\cdot 2^{-k})$.}
\label{dimentpicDist}
\end{figure}

By construction, the collection of balls
$$\cB_k := \{B_j = B(x_{Q_j},\delta 2^{-k}) : j = 1,\dots,L\}$$
is a $(\delta 2^{-k})$-packing of $B(x,2^{-k})$. Observe that $|x_{Q_j}-x|< \sqrt{d}\cdot 2^{-h-k-1}<2^{-k}$ by the choice of $h$ and that also $B_j\supset Q_j$ by \eqref{choiceofdelta}.
At this point, we define $K$ so that $K Q^{k+h,x}\supset B(x, 5\cdot 2^{-k})$, see Figure \ref{dimentpicDist}. For example, $K=11 \cdot 2^h$ works.
Hence, using also the definition of $\cH_k$,
$$\mu(B_j) \geq \mu(Q_j) > \epsilon_0\mu(Q^{k+h,x}) \stackrel{\eqref{2}}{\geq} \epsilon\mu(K Q^{k+h,x}) \geq \epsilon \mu(B(x,5 \cdot 2^{-k})).$$
Therefore
$\hom_{\delta,\epsilon,2^{-k}}(\mu,x) \geq \card \cB_k \ge 2^{at}/C_d \geq 2^{am} \geq \delta^{-m}$
for at least $pN$ values of $k \in [N]$, recall \eqref{choiceof_a} and \eqref{choiceofdelta}.
Theorem \ref{thm:homogeneity} is now proved for the points in $\Omega_s(\mu)$.

The last statement of Proposition \ref{prop:dyadichomo} yields for $\mu$ almost every $x \in \Omega^s(\mu)$ infinitely many $N \in \N$ such that \eqref{1} is satisfied for at least $p_0 N$ values $k \in [N]$. Furthermore, for $\mu$ almost every $x$ the property \eqref{2} still hold for these $N$, since it does not require any information about local dimension of $\mu$ at $x$. Hence the same proof goes through, and we have \eqref{hommu} for infinitely many $N \in \N$, as claimed.
\end{proof}

\subsection{Proof of Theorem \ref{thm:meanconical}}
Throughout this section we fix numbers $0<m<s<d$, where $m,d\in\N$, and $0<\alpha<1$. 

The main idea in the proof of Theorem \ref{thm:meanconical} is to combine the dyadic version of the homogeneity estimate from the previous subsection with the following lemma. This lemma
may be seen as a simple discrete version of Theorem \ref{thm:meanconical}.

\begin{lemma}
\label{lemma:cubesincones}
There exists $a_0 = a_0(m,s,d,\alpha) \in \N$ such that for any $a \geq a_0$ and any family $\cC$ of cubes of $Q\prec_a[0,1)^d$ with at least $2^{as}$ elements, there is $Q \in \cC$ such that for any $V \in G(d,d-m)$ and $\theta \in \S^{d-1}$ there exists $Q' \in \cC$ with
\begin{equation}\label{dcd}
Q' \subset X(x,V,\alpha) \setminus H(x,\theta,\alpha)
\end{equation}
for all $x \in Q$.
\end{lemma}

\begin{proof}
We denote by $C_1,\ldots,C_5$ positive and finite constants that depend only on $d,m$ and $\alpha$. We first let $C_1, C_2$ be constants such that (here $Q,Q'\in\mathcal{C}$ and $V\in G(d,d-m)$)
\begin{enumerate}
\item\label{yks} If $|x_Q-x_{Q'}|>C_1 2^{-a}$ and $V$ hits both $Q$ and $Q'$, then $x_{Q'}\in X(x_Q,V,\alpha/4)$.
\item\label{kaks} If $|x_Q-x_{Q'}|>C_1 2^{-a}$ and $x_{Q'}\in X(x_{Q},V,\alpha/4)$, then $Q'\subset X(y,V,\alpha/2)$ for all $y\in Q$.
\item\label{kolme} If $Q_1,\ldots,Q_{C_2}\prec_a[0,1)^d$ and the centers of $Q_i$ are $C_1 2^{-a}$ apart from each other, then there is $i_0\in[C_2]$ with the property that for each $\theta\in \S^{d-1}$, we have $Q_i\cap H(y,\theta,\alpha)=\emptyset$ for some $i\in[C_2]$ and for all $y\in Q_{i_0}$.
\end{enumerate}
The existence of such $C_1$ and $C_2$ is based on straightforward geometric arguments, see e.g. \cite{ErdosFuredi83}, \cite[Lemmas 2.1 and 2.3]{KaenmakiSuomala11} or \cite[Lemmas 4.2 and 4.3]{CKRS10}.

Since the Grassmannian $G(d,d-m)$ is compact, we may find $V_1,\ldots, V_{C_3}\in G(d,d-m)$ so that each $X(x,V,\alpha)$ contains some $X(x,V_i,\alpha/2)$. Choose a small $\varepsilon=\varepsilon(d,m,\alpha)>0$ to be determined later.
We say that a cube $Q$ in the collection $\cC$ is \emph{$i$-good}, if for all $\theta\in \S^{d-1}$ there is $Q'\in\cC$ such that
\[
Q'\subset X(x,V_i,\alpha/2)\setminus H(x,\theta,\alpha)\quad\text{for all }x\in Q.
\]
Next we estimate the amount of $i$-good cubes in $\cC$. We first choose an $2^{-a-1}$-dense subset $\{y_1,\ldots, y_{C_4 2^{am}}\}\subset \proj_{V_{i}^\perp}([0,1)^d)$ (where $\proj_{V_{i}^\perp}$ is the orthogonal projection onto $V_{i}^\perp$). Consider a $y_j$ for which the tube
\[T:=\{x\in\R^d\,:\,\dist(x,\proj_{V_i}^{-1}(y_j))<2^{-a-1}\}\]
contains at least $\varepsilon 2^{a(s-m)}$ center points of cubes in $\cC$, see Figure \ref{dimentpicTubeCubes} below.

\begin{figure}[h]
\begin{center}
\includegraphics[scale=0.3]{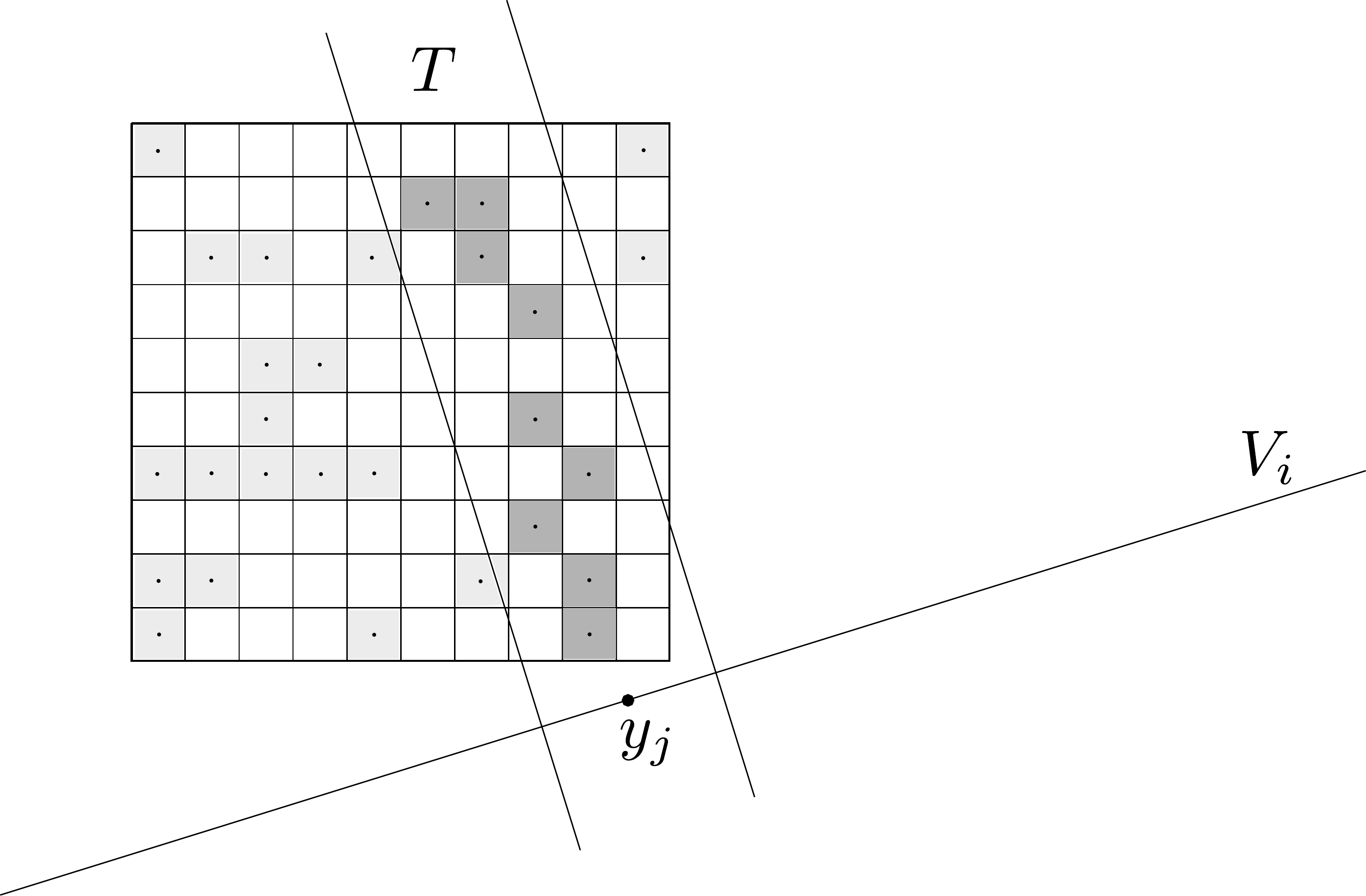}
\end{center}
\caption{The $y_j$ is located at $V_i$ such that the tube $T$ contains at least $\varepsilon 2^{a(s-m)}$ center points of cubes from the collection $\cC$.}
\label{dimentpicTubeCubes}
\end{figure}

Observe that, for all but at most $C_4 \varepsilon 2^{as}$ cubes $Q\in\cC$, the center $x_Q$ is contained in such a $T$. Let $\cC_{T}$ be the cubes in $\cC$ whose centers lie in $T$. For each collection of cubes $Q\prec_a[0,1)^d$ with $M$ elements, there is a sub-collection with at least $M/(2C_1+1)^d$ elements such that the midpoints are $C_1 2^{-a}$ apart from each other. We let $\cC_1$ be a maximal subcollection of $\cC_T$ with this property. We continue inductively; If the collection $\cC_T\setminus\cup_{k<n}\cC_k$ has at least $C_2(2C_1+1)^d/\varepsilon$ elements, we let $\cC_n$ be a maximal subcollection of $\cC_T\setminus\cup_{k<n}\cC_k$ with the central points $C_1 2^{-a}$ apart from each other. This process terminates when the cardinality of $\cC_T\setminus\cup_{k<n}\cC_k$ is at most $C_2(2C_1+1)^d/\varepsilon$ which, on the other hand,  is less than $\varepsilon$ times the cardinality of $\cC_T$ provided $2^{a(s-m)}\ge\varepsilon^{-3}C_2(2C_1+1)^d$ (this determines the choice of $a_0$). Furthermore, using \ref{yks}--\ref{kolme}, we see that all but $C_2$ of the elements of each $\cC_n$ are $i$-good and that $C_2$ is less than $\varepsilon$ times the cardinality of $\cC_n$.

To summarize, we have seen that a proportion at least $(1-C_4\varepsilon)$ of the cubes in $\cC$ belong to some $\cC_T$ for which there are at least $\varepsilon 2^{a(m-s)}$ elements and, among these, a proportion of at least $(1-\varepsilon)$ belong to a $\cC_n$ which has a proportion $(1-\varepsilon)$ of $i$-good cubes. Thus, at least a proportion $(1-C_5\varepsilon)$ of the cubes of $\cC$ are $i$-good, and this holds for all $i$. Whence, choosing $\varepsilon<1/(C_3 C_5)$, there is at least one cube $Q\in\cC$ which is $i$-good for all $V_1,\ldots,V_{C_3}$.
The claim \eqref{dcd} holds true for this $Q$.
\end{proof}

We can apply this lemma to gain information on the existence of trapped cubes:

\begin{definition}[Trapped cubes]
Let $a,k \in \N$, $Q \in \cQ$ and $\epsilon > 0$. A cube $Q' \prec_a Q$ with $\mu(Q') > \epsilon\mu(Q)$ is $\epsilon$-\textit{trapped} if, for any $V \in G(d,d-m)$ and $\theta \in \S^{d-1}$, there exists $Q'' \prec_a Q$ with $\mu(Q'') > \epsilon\mu(Q)$ and
$$Q'' \subset X(x,V,\alpha) \setminus H(x,\theta,\alpha)\quad\text{for all }x\in Q',$$
see Figure \ref{dimentpicTrappedCubes}. Once $\epsilon$ and $a$ are fixed, denote
$$Q_{\trap} := \bigcup_{Q' \prec_a Q \text{ is } \epsilon\text{-trapped}}Q'.$$
\end{definition}

\begin{lemma}
\label{lemma:conescales}
There exist $a=a(m,s,d,\alpha)$, $\epsilon = \epsilon(m,s,d,\alpha) > 0$, and $p' = p'(m,s,d,\alpha) > 0$ such that for $\mu$ almost every $x\in \Omega_s(\mu)$ and for all large enough $N \in \N$, the cube $Q^{k,x}$ is $\epsilon$-trapped for at least $p' N$ values $k \in [N]$. For $\mu$ almost every $x \in \Omega^s(\mu)$ there exist infinitely many $N \in \N$ such that the cube $Q^{k,x}$ is $\epsilon$-trapped for at least $p' N$ values of $k \in [N]$.
\end{lemma}

\begin{proof}
Let $t=(m+s)/2$ and $a = a_0(m,t,d,\alpha)$ be the constant from Lemma \ref{lemma:cubesincones} 
and $\epsilon = \epsilon_0(t,s,d,a)$ be the constant from Proposition \ref{prop:dyadichomo}. Finally, let $p' = \epsilon p_0/2$, where $p_0 = p_0(t,s,d)$ is the constant from Proposition \ref{prop:dyadichomo}. Let $x \in \R^d$ and $k \in \N$ such that
\begin{equation}\label{traphom}
\hom^a_{\epsilon,k}(\mu,x) \geq 2^{at}.
\end{equation}
Then by Lemma \ref{lemma:cubesincones} there exists an $\epsilon$-trapped cube $Q \prec_a Q^{k,x}$. 
Proposition \ref{prop:dyadichomo} implies that, for $\mu$ almost every $x \in \Omega_s(\mu)$ and for all large enough $N \in \N$, \eqref{traphom} holds for at least $p_0 N$ values $k \in [N]$. For these $x$ and $N$, we have
\begin{equation}\label{sumtrap}
\sum_{k = 1}^N \mu_{k,x}(Q^{k,x}_{\trap}) > \epsilon p_0 N = 2 p'N.
\end{equation}
On the other hand, Lemma \ref{lemma:blackandwhite} with 'black' = 'trapped' implies for $\mu$ almost every $x \in \R^d$ that
$$\lim_{N \to \infty} \frac{1}{N} \sum_{k = 1}^N \left(\mu_{k,x}(Q^{k,x}_{\trap}) - \1_{Q^{k,x} \text{ is trapped}}\right) = 0.$$
We conclude that for $\mu$ almost every $x \in \Omega_s(\mu)$, and for all large enough $N \in \N$, we have
$$\sum_{k = 1}^N \1_{Q^{k,x} \text{ is trapped}} > p'N,$$
which is precisely what we wanted.

For the points in $\Omega^s(\mu)$ the argument is symmetric as Proposition \ref{prop:dyadichomo} implies \eqref{sumtrap} for infinitely many $N$.
\end{proof}

\begin{figure}[h]
\begin{center}
\includegraphics[scale=0.5]{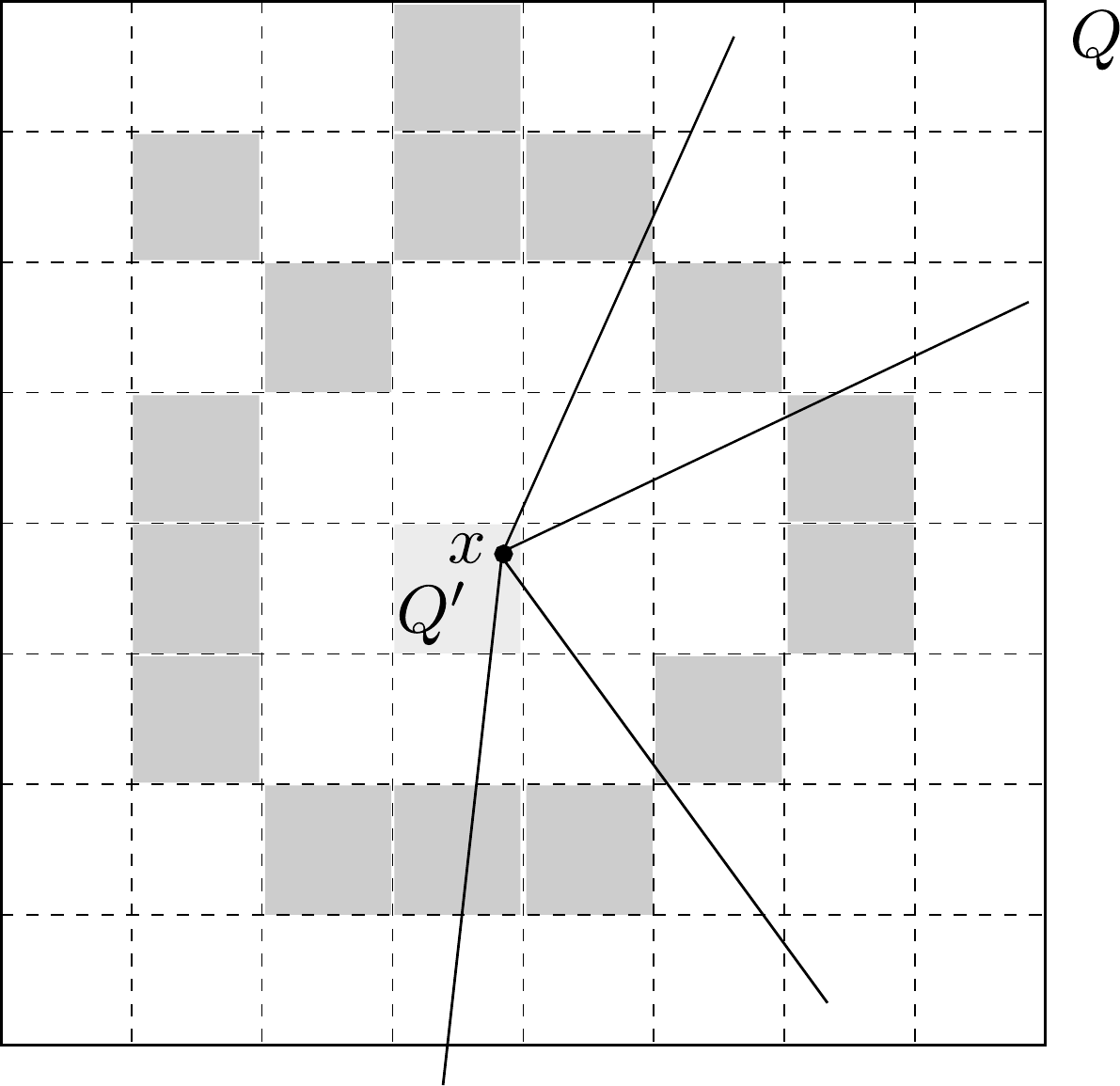}
\end{center}
\caption{A cube $Q' \prec_a Q$ becomes $\epsilon$-trapped if it is effectively surrounded by cubes $Q'' \prec_a Q$ (in the picture the darker gray cubes) each of which having at least $\epsilon$-portion of $\mu$ mass in $Q$.}
\label{dimentpicTrappedCubes}
\end{figure}

\begin{proof}[Proof of Theorem \ref{thm:meanconical}.] Let $h$ be the smallest integer such that $2^h \geq \sqrt{d}$, and let $K=K(d)$ be large enough so that $B(x,2^{-k}) \subset K Q^{k+h,x}$. For example, $K=2^{h+1}+1$ works.
Let $a=a(m,s,d,\alpha)$, $\epsilon = \epsilon(m,s,d,\alpha)$, and $p' = p'(m,s,d,\alpha)$ be the constants from Lemma \ref{lemma:conescales}. Define
$p = p'/2.$
and $q=1-p'/2$. As in the proof of Theorem \ref{thm:homogeneity}, we may assume that $\mu$ has been suitably localized and translated to ensure that, applying Lemma \ref{lemma:meandoubling} with the appropriate parameters, there is a constant $c_q>0$ such that
\begin{equation} \label{eq:modified-mean-doubling}
\liminf_{N\rightarrow\infty}\frac1N  \card\{ k \in [N] : \mu(Q^{k,x}) \geq c_q\mu(K Q^{k,x})\} > q\quad\text{for }\mu \text{ almost all } x.
\end{equation}
We also set
$$c = c(m,s,d,\alpha) = c_q\epsilon.$$

Consider $x \in \R^d$ and $k \in \N$ for which
\begin{align}
\label{trappeddoubling}
Q^{k+h+a,x} \text{ is } \epsilon\text{-trapped} \quad \text{and} \quad \mu(Q^{k+h,x}) \geq c_q\mu(K Q^{k+h}).
\end{align}
Let $V \in G(d,d-m)$ and $\theta \in \S^{d-1}$. Since $Q^{k+h+a,x}$ is $\epsilon$-trapped, we can choose $Q \prec_a Q^{k+h,x}$ with $\mu(Q) > \epsilon\mu(Q^{k+h,x})$ and
$$Q \subset X(x,V,\alpha) \setminus H(x,\theta,\alpha).$$
The choice of $h$ implies that $Q \subset Q^{k+h,x} \subset B(x,2^{-k})$, whence also
$$Q \subset X(x,2^{-k},V,\alpha) \setminus H(x,\theta,\alpha).$$
Moreover, using \eqref{trappeddoubling} and that $B(x,2^{-k}) \subset K Q^{k+h,x}$ (by the choice of $K$), we have
$$\mu(B(x,2^{-k})) \leq \mu(K Q^{k+h,x}) \leq c_q^{-1} \mu(Q^{k+h,x}).$$ Hence
\begin{align*}
\frac{\mu(X(x,2^{-k},V,\alpha) \setminus H(x,\theta,\alpha))}{\mu(B(x,2^{-k}))} \geq \frac{\mu(Q)}{c_q^{-1}\mu(Q^{k+h,x})} > c_q\epsilon = c.
\end{align*}
Lemma \ref{lemma:conescales} and \eqref{eq:modified-mean-doubling} now imply that, for $\mu$ almost every $x \in \Omega_s(\mu)$ and for all large enough $N \in \N$, the two properties in \eqref{trappeddoubling} are simultaneously satisfied for at least $(p'+q-1)N = pN$ values of $k \in [N]$. Hence the above argument implies the claim for $\mu$ almost every $x \in \Omega_s(\mu)$.

As for the points in $\Omega^s(\mu)$, the second part of Lemma \ref{lemma:conescales} implies the claim for infinitely many $N$ in a similar manner.
\end{proof}

\subsection{Proof of Theorem \ref{thm:meanporosity}}

Throughout this section integers $d\in\N$, $\ell\in[d]$, a number $0<\alpha<1/2$ and a measure $\mu$ on $[0,1)^d$
are fixed.

\begin{definition}[Porous cubes]
Choose the unique $a = a(\alpha)\in \N$ such that
\begin{align}
\label{choiceofa}
2^{-a} \leq 1-2\alpha < 2^{1-a}.
\end{align}
Given $\epsilon > 0$, a cube $Q'\in\mathcal{Q}_{k+a}$ is $(\alpha,\epsilon)$-\textit{porous}, if there exists $x\in Q'$ such that
$$\por_\ell(\mu,x,2^{-k},\varepsilon)>\alpha,$$
see Figure \ref{dimentpicPorousCubes}. We often suppress the notation $(\alpha,\epsilon)$ from the definition of porous cubes if they are clear from the context. The first such occasion is at hand: if $Q \in \cQ_k$ we write
$$Q_{\por} := \bigcup_{Q'\prec_a Q \text{ is porous}} Q'.$$
\end{definition}

Recall the definitions of $P_{\ell,\alpha,p}$ and $U_{\ell,\alpha,p}$ from Definition \ref{porodefi}.
\begin{lemma}
\label{porouslargenumbers}
Fix $\epsilon > 0$ and $0<p'< p < 1$. Then for any measure $\mu$ on $\R^d$ and for $\mu$ almost every $x\in P_{\ell,\alpha,p}$, we have
\begin{equation}\label{Pdef}
p_N(x) := \frac{1}{N}\sum_{k=1}^{N} \mu_{k,x}(Q^{k,x}_{\por}) \geq p'
\end{equation}
when $N \in \N$ is large. For $\mu$ almost all $x\in U_{\ell,\alpha,p}$ the estimate \eqref{Pdef} holds for infinitely many $N$.
\end{lemma}

\begin{proof}If $\por_\ell (\mu,x,2^{-(k-a)},\epsilon) > \alpha$, then $Q^{k,x}$ is porous. If $x \in P_{\ell,\alpha,p}$ (resp. $U_{\ell,\alpha,p}$) then, for large enough $N$ (infinitely many $N$), this happens for at least $pN$ of indices $k \in [N]$, i.e.
$$\liminf_{N \to \infty} \frac{1}{N} \sum_{k = 1}^N \1_{Q^{k,x} \text{ is porous}} \geq p \quad \text{for all } x \in P_{\ell,\alpha,p}$$
(resp. $\limsup_{N\to\infty}\cdots\geq p$ for all $x \in U_{\ell,\alpha,p}$).
Lemma \ref{lemma:blackandwhite} with 'black' = 'porous' yields the claim.
\end{proof}

\begin{figure}[h]
\begin{center}
\includegraphics[scale=0.3
]{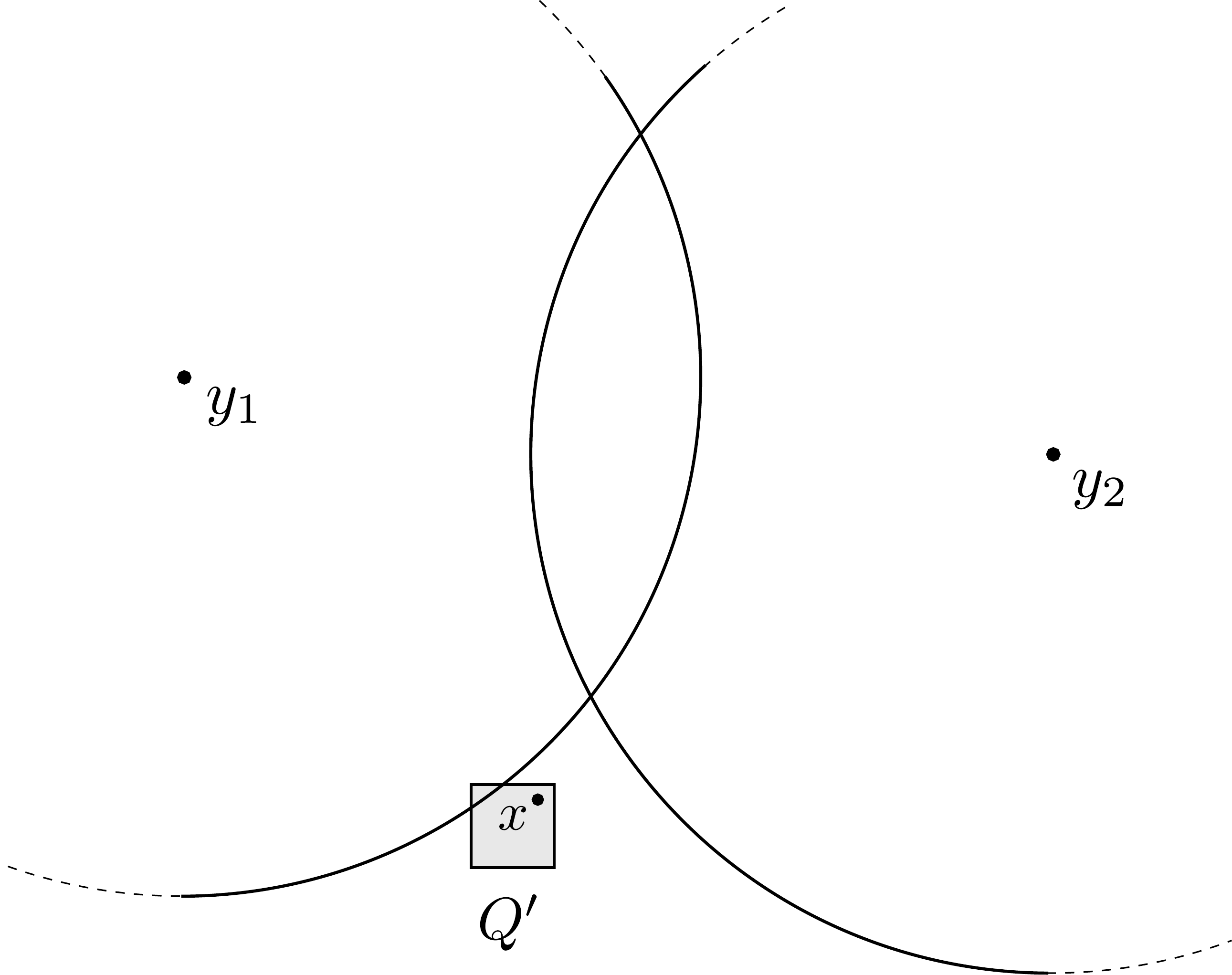}
\end{center}
\caption{In the picture $d = 2$, $\ell = 2$ and $\alpha$ is close to $1/2$. A cube $Q' \in \cQ_{k+a}$ becomes $(\alpha,\epsilon)$-porous if it contains a point $x$ such that the $\ell$-porosity at $x$ is at least $\alpha$. Since $\alpha$ is very close to $1/2$, the number $a = a(\alpha)$ is very large, so the size of the porous cube $Q'$ is very small compared to the size of the holes $B(y_1,\alpha 2^{-k})$ and $B(y_2,\alpha 2^{-k})$ obtained from the $\ell$-porosity.}
\label{dimentpicPorousCubes}
\end{figure}

Recall the following covering lemma \cite[Lemma 5.4]{KRS11}:
\begin{lemma}
\label{lemma:kporouscover}
Let $x \in \R^d$, $r > 0$ and $A \subset B(x,r)$. If $\por_\ell(A,z,r) \geq \alpha$ for every $z \in A$, then $A$ can be covered with $c(1 - 2\alpha)^{\ell-d}$ balls of radius $(1 - 2\alpha)r$, where $c = c(d) > 0$.
\end{lemma}

With this we obtain the following lemma, which is analogous to \cite[Lemma 3.5]{BJJKRSS09}:
\begin{lemma}
\label{lemma:decompositionporous}
Let $k \in \N$, $\epsilon > 0$ and $Q \in \cQ_k$. Then $Q$ may be divided into three disjoint parts
$$Q=E\cup P\cup J,$$
where
\begin{enumerate}
\item $\mu(E)\leq c_0\varepsilon \mu(3Q)$,
\item $P$ can be covered by at most $c_1 2^{a(d-\ell)}$ cubes $Q'\prec_a Q$, and
\item   $J\cap Q_{\por}=\emptyset$.
\end{enumerate}
 Here $c_0=c_0(\ell,\alpha,d)$ and $c_1=c_1(d)$ are positive and finite constants.
\end{lemma}

\begin{proof} For any porous $Q' \prec_a Q$, choose $x = x_{Q'} \in Q'$ such that
$\por_\ell(\mu,x,2^{-k},\epsilon) > \alpha$.
By the definition of $\ell$-porosity, this implies that we may choose points $y_1,\dots,y_\ell \in \R^d$ with $(y_i-x) \cdot (y_j-x) = 0$ for $i \neq j$ such that the balls
$B_{Q'}^j := B(y_j,\alpha 2^{-k})$ 
satisfy
$B_{Q'}^j \subset B(x,2^{-k})$ and $\mu(B_{Q'}^j) \leq \epsilon\mu(B(x,2^{-k}))$ for $j\in[\ell]$.
Since $B(x,2^{-k}) \subset 3Q$ (because $x \in Q' \subset Q$), this gives
\begin{align}
\label{porousball}
B_{Q'}^j \subset 3Q \quad \text{and} \quad \mu(B_{Q'}^j) \leq \epsilon\mu(3 Q), \quad j\in[\ell].
\end{align}
Denote
$$E := Q \cap \bigcup_{Q' \prec_a Q \text{ is porous}\atop j\in[\ell]} B_{Q'}^j.$$
By \eqref{porousball} we have
$$\mu(E) \leq c_0 \epsilon \mu(3Q),$$
where $c_0 = c_0(\ell,\alpha,d) := \ell 2^{ad}$. Now we can define
$P := Q_{\por} \setminus E \quad \text{ and } \quad J := Q \setminus (E \cup P)$.
If $z\in P$, there is $x\in P$ with $|x-z|<\sqrt{d}2^{-a-k}$ such that $\por_\ell(P,x,2^{-k})>\alpha$. An easy calculation then implies that $\por_\ell(P,z,2^{-k})\ge\alpha'$ for $\alpha'=\alpha-\sqrt{d}2^{-a}$. Since $Q$ can be covered by $C(d)$ balls of radius $2^{-k}$, the Lemma \ref{lemma:kporouscover} yields $c=c(d)<\infty$ and a collection $\cB$ of at most $c(1-2\alpha')^{\ell-d}$ balls of radius $(1-2\alpha')2^{-k}$ whose union cover $P$. Since
$$1-2\alpha'=1-2\alpha+\sqrt{d}2^{1-a}\le (1+\sqrt{d})2^{1-a}$$
by \eqref{choiceofa}, it follows that $P$ may be covered by $c_1 2^{a(d-\ell)}$ cubes $Q'\prec_a Q$ for some $c_1=c_1(d)<\infty$.

Finally, the claim for $J$ is immediate from its definition.
\end{proof}

\begin{proof}[Proof of Theorem \ref{thm:meanporosity}]  Let $0<q<1$. By localizing and normalizing, we can assume that $\mu$ is a probability measure supported on $[0,1)^d$ and, thanks to Lemma \ref{lemma:meandoubling} and a random translation, also that
\begin{equation} \label{eq:doubling-scales}
\liminf_{N\rightarrow\infty}\frac1N \card\{ k \in [N] : \mu(Q^{k,x}) \geq \beta(q)\mu(3Q^{k,x})\} > q\quad\text{ for }\mu\text{ almost all } x,
\end{equation}
where $0<\beta(q)<1/(2 c_0)$ is a function of $q$ (which also depends on the dimension $d$) such that
\begin{equation*}
\lim_{g\uparrow 1}\beta(q)=0.
\end{equation*}

Let $\epsilon=\epsilon(q) := \beta(q)^2$. When $x \in [0,1)^d$ and $k \in \N$, let $Q^{k,x} = E^{k,x} \cup P^{k,x} \cup J^{k,x}$ be the decomposition given by Lemma \ref{lemma:decompositionporous} with this $\epsilon$. Write $K = c_1 2^{a(d-\ell)}$ and let $\widetilde{P}^{k,x}$ be the union of the at most $K$ cubes that cover $P^{k,x}$ from Lemma \ref{lemma:decompositionporous}. Now
$$Q_{\por}^{k,x} \setminus \widetilde{P}^{k,x} \subset E^{k,x}, \quad P^{k,x} \subset \widetilde{P}^{k,x}, \quad \text{and} \quad J^{k,x} \subset Q^{k,x} \setminus Q^{k,x}_{\por}.$$
Recall that $\varphi(t)=t\log(1/t)$ is the entropy function. Given $k\in\N$, we can estimate
$$H^a(\mu,Q^{k,x}) \leq \sum_{Q \prec_a Q^{k,x}\atop Q \subset Q_{\por}^{k,x} \setminus \widetilde{P}^{k,x}} \varphi(\mu_{k,x}(Q)) + \sum_{Q \prec_a Q^{k,x}\atop Q \subset Q_{\por}^{k,x} \cap \widetilde{P}^{k,x}} \varphi(\mu_{k,x}(Q)) + \sum_{Q \prec_a Q^{k,x}\atop Q \cap Q_{\por}^{k,x} = \emptyset} \varphi(\mu_{k,x}(Q)).$$
Let $S^{k,x}_1$, $S^{k,x}_2$ and $S^{k,x}_3$ be the three sums in the right-hand side above. Using the fact that $Q^{k,x}_{\por} \setminus \widetilde{P}^{k,x} \subset E^{k,x}$ and the definition of $\widetilde{P}^{k,x}$, Lemma \ref{trivialentropy} yields
\begin{align*}
S^{k,x}_1 &\leq \mu_{k,x}(E^{k,x}) \log \frac{2^{ad}}{\mu_{k,x}(E^{k,x})}, \\
S^{k,x}_2 &\leq \mu_{k,x}(Q^{k,x}_{\por}) \log \frac{K}{\mu_{k,x}(Q^{k,x}_{\por})},\\
S^{k,x}_3 &\leq (1-\mu_{k,x}(Q^{k,x}_{\por})) \log \frac{2^{ad}}{1-\mu_{k,x}(Q^{k,x}_{\por})}.
\end{align*}
For $N \in \N$, denote
$D_N(x) = \{k \in [N] : \mu(Q^{k,x}) \geq \beta(q) \mu(3Q^{k,x})\},$
and let $q_N(x) = \card D_N(x) / N$.

Since
$$\mu(E^{k,x}) \leq c_0 \epsilon\mu(3Q^{k,x}) \leq c_0 \beta(q) \mu(Q^{k,x})$$
for all $k \in D_N(x)$, we have
\begin{equation}\label{aa}
\sum_{k \in D_N(x)} S^{k,x}_1 \leq c_0 \beta(q) N \log \frac{2^{ad}}{c_0 \beta(q)}.
\end{equation}
Moreover, the log sum inequality (Lemma \ref{lemma:logsum}) implies
\begin{equation}\label{bee}
\sum_{k = 1}^N S^{k,x}_2 + \sum_{k = 1}^N S^{k,x}_3 \leq p_N(x) N\log \frac{K}{p_N(x)} + (1-p_N(x))N\log \frac{2^{ad}}{1-p_N(x)},
\end{equation}
where
$$p_N(x) = \frac{1}{N}\sum_{k = 1}^N \mu_{k,x}(Q^{k,x}_{\por}).$$

We are going to use the above estimates on $S^{k,x}_1$ for the ``doubling'' scales $k\in D_N(x)$. For other values of $k$, we do not have any control over $\mu_{k,x}(E^{k,x})$, so we use the trivial bound (recall Lemma \ref{trivialentropy})
\begin{equation}\label{cee}
H^a(\mu,Q^{k,x})\le ad.
\end{equation}
As the amount of non-doubling scales can be made arbitrarily small by letting $q\to 1$, this will cause no harm in the end.

Combining \eqref{aa}, \eqref{bee}, \eqref{cee}, and plugging in $K = c_1 2^{a(d-\ell)}$, we have proven that for each $x \in \R^d$ and $N \in \N$,
\begin{align*}
\frac{1}{Na}\sum_{k = 1}^N H^a(\mu,Q^{k,x}) \leq d-p_N(x)\ell+\frac{c_2(q,x,N)}{a}+(1-q_N(x))d+c_0 d \beta(q),
\end{align*}
where
\begin{align*}
c_2(q,x,N) = p_N(x)\log c_1+p_N(x)\log\tfrac{1}{p_N(x)}+(1-p_N(x))\log\tfrac{1}{1-p_N(x)} +  c_0 \beta(q)\log\tfrac{1}{c_0 \beta(q)}
\end{align*}
satisfies $c_2(q,x,N)\le c_3<\infty$ for some uniform constant $c_3$.
If we now combine Lemma \ref{porouslargenumbers} and \eqref{eq:doubling-scales}, we get that for all $0<p'<p$, for $\mu$ almost every $x \in P_{\ell,\alpha,p}$, and for all large enough $N \in \N$,
$$p_N(x) \geq p' \quad \text{and} \quad q_N(x) \geq q.$$
Hence local entropy averages implies (letting $p'\rightarrow p$)
\begin{align*}
\ulocd(\mu,x) &= \limsup_{N \to \infty}\frac{1}{Na}\sum_{k = 1}^N H^a(\mu,Q^{k,x})\\
 &\leq d-p\ell+\frac{c_3}{a}+(1-q+c_0\beta(q))d
 \end{align*}
for $\mu$ almost every $x \in P_{\ell,\alpha,p}$. For $\mu$ almost all $x\in U_{\ell,\alpha,p}$ we have $p_N(x)\ge p$ for infinitely many $N\in\N$ and this leads to
\begin{align*}
\llocd(\mu,x) \leq d-p\ell+\frac{c_3}{a}+(1-q+c_0\beta(q))d.
 \end{align*}
for almost all $x\in U_{\ell,\alpha,p}$.
Since $2^a \geq \tfrac{1}{1-2\alpha}$ by \eqref{choiceofa}, and $1-q+c_0\beta(q)\rightarrow 0$ as $q \to 1$, the required estimates follow.
\end{proof}

\section{Further remarks}
\label{sec:remarksandfurtheresults}

We discuss here some of the questions raised by our results. It seems likely that at least some of them could be answered by further developing the technique of local entropy averages. On the other hand, some of them may turn out to be harder and require deeper new ideas.

\emph{1)} Although we do not make them explicit, all the constants appearing in Theorems \ref{thm:homogeneity},\ref{thm:meanconical} and \ref{thm:meanporosity} are effective. However, in all cases they are very far from optimal. In particular, they worsen very fast with the ambient dimension $d$; it would be interesting to know if this a genuine phenomenon or an artifact from the method (in particular, the switching between balls and cubes).

\emph{2)} As mentioned above, the estimate of Theorem \ref{thm:meanporosity} is asymptotically sharp as $\alpha\rightarrow\tfrac12$. See e.g. \cite[Example 3.9]{BJJKRSS09}. When $\alpha$ is fixed and $p\rightarrow 0$, the proof of Theorem \ref{thm:meanporosity} gives the following sharper estimate
\begin{align*}
\udimloc(\mu,x)\leq d-p\ell+\frac{c(d)p\log{\frac1p}}{\log\frac{1}{1-2\alpha}}
\end{align*}
for $\mu$ almost all $x\in P_{\ell,\alpha,p}(\mu)$ (and similarly with $\ldimloc$ and $U_{\ell,\alpha,p}$). We believe this estimate should hold without the $\log\tfrac1p$ term, i.e. that
\begin{align*}
\udimloc(\mu,x)\leq d-p\ell+\frac{c(d)p}{\log\frac{1}{1-2\alpha}}\,.
\end{align*}

\emph{3)} We believe it should be possible to choose $p$ independent of $\alpha$ in Theorem \ref{thm:meanconical}, even tough our proof does not imply this. In the homogenity estimate (Theorem \ref{thm:homogeneity}) and its dyadic counterpart (Proposition \ref{prop:dyadichomo}), $p$ is independent of $\delta$ (resp. $a$), but it is easy to see that the statement fails as $p\rightarrow 1$. Also, it is not known what are the correct asymptotics for $c$ in Theorem \ref{thm:meanconical} as $\alpha\rightarrow 0$ and/or $s\rightarrow m$.

\emph{4)} Koskela and Rohde \cite{KoskelaRohde97} consider a version of mean porosity for sets that contain holes for a fixed proportion of $n\in\N$ in the annuli $A_n=B(x,\lambda^{n})\setminus B(x,\lambda^{n+1})$. They obtain the sharp bound for the packing dimension of such mean porous sets as the parameter $\lambda\rightarrow 1$. In \cite{Nieminen06} Nieminen is interested in a weak form of porosity where the relative size of the ``pores'' is allowed to go to zero when $r\to 0$. It seems possible that the method of local entropy averages can be used to provide measure versions of their results. Concerning the porosity condition in \cite{Nieminen06}, one should obtain a gauge function $h$ depending on the porosity data such that a porous measure $\mu$ is absolutely continuous with respect to the Hausdorff or packing measure in this gauge. Our method can certainly be used to obtain dimension estimates for spherically porous sets and measures, see e.g. \cite{KST00}.

\emph{5)} There are very few results on the dimension of porous sets or measures in metric spaces (see \cite{JJKRRS10,KRS11}). It would be very interesting to see, if one could apply a version of the local entropy averages to obtain new dimension bounds in this direction. If one applies the local entropy averages formula directly, there is usually an extra error term arising from the geometry of the metric  dyadic grid (the ``cubes'' are no longer isometric). For this reason, the straightforward generalisation of the method gives useful information only if the effect on entropy averages is larger than the error caused by this irregularity.

\section{Acknowledgements}

The authors are grateful to Antti K\"aenm\"aki for inspiring discussions related to this project.

\bibliographystyle{plain}
\bibliography{geoentFINAL}

\begin{thebibliography}{10}

\bibitem{BJJKRSS09}
D.~Beliaev, E.~J\"{a}rvenp\"{a}\"{a}, M.~J\"{a}rvenp\"{a}\"{a},
  A.~K\"{a}enm\"{a}ki, T.~Rajala, S.~Smirnov, and V.~Suomala.
\newblock Packing dimension of mean porous measures.
\newblock {\em J. London Math. Soc.}, 80(2):514--530, 2009.

\bibitem{BeliaevSmirnov02}
D.~Beliaev and S.~Smirnov.
\newblock On dimension of porous measures.
\newblock {\em Math. Ann.}, 323(1):123--141, 2002.

\bibitem{CKRS10}
M.~Cs{\"o}rnyei, A.~K{\"a}enm{\"a}ki, T.~Rajala, and V.~Suomala.
\newblock Upper conical density results for general measures on {$\Bbb R^n$}.
\newblock {\em Proc. Edinb. Math. Soc. (2)}, 53(2):311--331, 2010.

\bibitem{Dolzenko67}
E.~P. Dol{\v{z}}enko.
\newblock Boundary properties of arbitrary functions.
\newblock {\em Izv. Akad. Nauk SSSR Ser. Mat.}, 31:3--14, 1967.

\bibitem{EJJ00}
J.-P. Eckmann, E.~J{\"a}rvenp{\"a}{\"a}, and M.~J{\"a}rvenp{\"a}{\"a}.
\newblock Porosities and dimensions of measures.
\newblock {\em Nonlinearity}, 13(1):1--18, 2000.

\bibitem{ErdosFuredi83}
P.~Erd\"os and Z.~F\"uredi.
\newblock The greatest angle among {$n$} points in the {$d$}-dimensional
  {E}uclidean space.
\newblock {\em North-Holland Math. Stud.}, 75:275--283, 1983.

\bibitem{Feller71}
W.~Feller.
\newblock {\em An introduction to probability theory and its applications.
  {V}ol. {II}.}
\newblock Second edition. John Wiley \& Sons Inc., New York, 1971.

\bibitem{Hochman11}
M.~Hochman.
\newblock Private communication.
\newblock 2011.

\bibitem{HochmanShmerkin11}
M.~Hochman and P.~Shmerkin.
\newblock Local entropy averages and projections of fractal measures.
\newblock {\em Ann. of Math. (2)}, 175(3):1001--1059, 2012.

\bibitem{Jarvenpaa10}
E.~J{\"a}rvenp{\"a}{\"a}.
\newblock Dimensions and porosities.
\newblock In {\em Recent developments in fractals and related fields}, Appl.
  Numer. Harmon. Anal., pages 35--43. Birkh\"auser Boston Inc., Boston, MA,
  2010.

\bibitem{JarvenpaaJarvenpaa05}
E.~J{\"a}rvenp{\"a}{\"a} and M.~J{\"a}rvenp{\"a}{\"a}.
\newblock Average homogeneity and dimensions of measures.
\newblock {\em Math. Ann.}, 331(3):557--576, 2005.

\bibitem{JJKRRS10}
E.~J{\"a}rvenp{\"a}{\"a}, M.~J{\"a}rvenp{\"a}{\"a}, A.~K{\"a}enm{\"a}ki,
  T.~Rajala, S.~Rogovin, and V.~Suomala.
\newblock Packing dimension and {A}hlfors regularity of porous sets in metric
  spaces.
\newblock {\em Math. Z.}, 266(1):83--105, 2010.

\bibitem{Kaenmaki10}
A.~K{\"a}enm{\"a}ki.
\newblock On upper conical density results.
\newblock In {\em Recent developments in fractals and related fields}, Appl.
  Numer. Harmon. Anal., pages 45--54. Birkh\"auser Boston Inc., Boston, MA,
  2010.

\bibitem{KRS11}
A.~K{\"a}enm{\"a}ki, T.~Rajala, and V.~Suomala.
\newblock Local homogeneity and dimensions of measures in doubling metric
  spaces.
\newblock 2010.
\newblock Preprint available at {\tt http://arxiv.org/abs/1003.2895}.

\bibitem{KaenmakiSuomala08}
A.~K{\"a}enm{\"a}ki and V.~Suomala.
\newblock Conical upper density theorems and porosity of measures.
\newblock {\em Adv. Math.}, 217(3):952--966, 2008.

\bibitem{KaenmakiSuomala11}
A.~K{\"a}enm{\"a}ki and V.~Suomala.
\newblock Nonsymmetric conical upper density and {$k$}-porosity.
\newblock {\em Trans. Amer. Math. Soc.}, 363(3):1183--1195, 2011.

\bibitem{KoskelaRohde97}
P.~Koskela and S.~Rohde.
\newblock Hausdorff dimension and mean porosity.
\newblock {\em Math. Ann.}, 309(4):593--609, 1997.

\bibitem{KST00}
P.~Koskela, N.~Shanmugalingam, and H.~Tuominen.
\newblock Removable sets for the {P}oincar\'e inequality on metric spaces.
\newblock {\em Indiana Univ. Math. J.}, 49(1):333--352, 2000.

\bibitem{LN04}
J.~G. Llorente and A.~Nicolau.
\newblock Regularity properties of measures, entropy and the law of the
  iterated logarithm.
\newblock {\em Proc. London Math. Soc.}, 89(3):485--524, 2004.

\bibitem{Lorent03}
A.~Lorent.
\newblock A generalised conical density theorem for unrectifiable sets.
\newblock {\em Ann. Acad. Sci. Fenn. Math.}, 28(2):415--431, 2003.

\bibitem{Marstrand54}
J.~M. Marstrand.
\newblock Some fundamental geometrical properties of plane sets of fractional
  dimensions.
\newblock {\em Proc. London Math. Soc. (3)}, 4:257--302, 1954.

\bibitem{Mattila88}
P.~Mattila.
\newblock Distribution of sets and measures along planes.
\newblock {\em J. London Math. Soc. (2)}, 38(1):125--132, 1988.

\bibitem{Mattila95}
P.~Mattila.
\newblock {\em Geometry of sets and measures in {E}uclidean spaces: Fractals
  and rectifiability}, volume~44 of {\em Cambridge Studies in Advanced
  Mathematics}.
\newblock Cambridge University Press, Cambridge, 1995.

\bibitem{MattilaParamonov95}
P.~Mattila and P.~V. Paramonov.
\newblock On geometric properties of harmonic {${\rm Lip}_1$}-capacity.
\newblock {\em Pacific J. Math.}, 171(2):469--491, 1995.

\bibitem{NTV03}
F.~Nazarov, S.~Treil, and A.~Volberg.
\newblock The {$Tb$}-theorem on non-homogeneous spaces.
\newblock {\em Acta Math.}, 190(2):151--239, 2003.

\bibitem{Nieminen06}
T.~Nieminen.
\newblock Generalized mean porosity and dimension.
\newblock {\em Ann. Acad. Sci. Fenn. Math.}, 31(1):143--172, 2006.

\bibitem{Rajala09a}
T.~Rajala.
\newblock {\em Porosity and dimension of sets and measures}.
\newblock PhD thesis, University of Jyv\"askyl\"a,
  http://www.math.jyu.fi/research/reports/rep119.pdf, 2009.

\bibitem{Salli85}
A.~Salli.
\newblock Upper density properties of {H}ausdorff measures on fractals.
\newblock {\em Ann. Acad. Sci. Fenn. Ser. A I Math. Dissertationes}, (55):49,
  1985.

\bibitem{Shmerkin11b}
P.~Shmerkin.
\newblock Porosity, dimension, and local entropies: a survey.
\newblock {\em Rev. Un. Mat. Argentina}, 52(3):81--103, 2011.

\bibitem{Shmerkin11}
P.~Shmerkin.
\newblock The dimension of weakly mean porous measures: a probabilistic
  approach.
\newblock {\em Int. Math. Res. Not. IMRN}, (9):2010--2033, 2012.

\bibitem{Suomala05}
V.~Suomala.
\newblock On the conical density properties of measures on {${\Bbb R}^n$}.
\newblock {\em Math. Proc. Cambridge Philos. Soc.}, 138(3):493--512, 2005.

\bibitem{Vaisala87}
J.~V{\"a}is{\"a}l{\"a}.
\newblock Porous sets and quasisymmetric maps.
\newblock {\em Trans. Amer. Math. Soc.}, 299(2):525--533, 1987.

\end{thebibliography}

\end{document}